\documentclass[10pt,letterpaper]{article}
\usepackage[letterpaper, margin=.75in]{geometry}
\usepackage[latin1]{inputenc}
\usepackage{graphicx}
\graphicspath{{images/}{../images/}}
\usepackage{float}
\usepackage{xypic}
\usepackage{amsthm, amsmath, amssymb}
\usepackage{ mathrsfs }
\usepackage{marginnote}
\usepackage{multicol}
\usepackage{subfiles}

\usepackage{tikz, pgf, calc}
\usetikzlibrary{arrows,matrix,positioning,fit,calc}
\usepackage[linewidth=1pt]{mdframed}

\usepackage{parskip}

\usepackage{tikz}

\usepackage{cancel}

\newtheorem{theorem}{Theorem}[section]
\theoremstyle{definition}

\theoremstyle{definition}

\newtheorem{cor}{Corollary}[section]
\newtheorem{lem}{Lemma}[section]
\newtheorem*{rem}{Remark}

\newcommand{\R}{\mathbb{R}}

\newcommand{\ukone}{u^{k+1}}
\newcommand{\uk}{u^k}

\newcommand{\ekone}{e^{k+1}}
\newcommand{\ek}{e^k}

\newcommand{\tkone}{T^{k+1}}

\newcommand{\etkone}{e_T^{k+1}}

\author{Elizabeth Hawkins\thanks{\small
	School of Mathematical and Statistical Sciences, Clemson University, Clemson, SC, 29364, evhawki@clemson.edu.}}
\title{Accelerating convergence of a natural convection solver by continuous data assimilation}

\begin{document}
\maketitle
\begin{abstract}
The Picard iteration for the Boussinesq model of natural convection can be an attractive solver because it stably decouples the fluid equations from the temperature equation (for contrast, the Newton iteration does not stably decouple).  However, the convergence of Picard for this system is only linear and slows as the Rayleigh number increases, eventually failing for even moderately sized Rayleigh numbers.  We consider this solver in the setting where sparse solution data is available, e.g. from data measurements or solution observables, and enhance Picard by incorporating the data into the iteration using a continuous data assimilation (CDA) approach.  We prove that our approach scales the linear convergence rate by $H^{1/2}$, where $H$ is the characteristic spacing of the measurement locations. This implies that when Picard is converging, CDA will accelerate convergence, and when Picard is not converging, CDA (with enough data) will enable convergence.  In the case of noisy data, we prove that the linear convergence rate of the nonlinear residual is similarly scaled by $H^{1/2}$ but the accuracy is limited by the accuracy of the data. Several numerical tests illustrate the effectiveness of the proposed method, including when the data is noisy. These tests show that CDA style nudging adapted to an iteration (instead of a time stepping scheme) enables convergence at much higher $Ra$. 
\end{abstract}

\section{Introduction}
Fluid flows with varying temperature (or density) can be modeled by the Boussinesq system for non-isothermal flow/natural convection. 
This multiphysics model can be observed in nature in atmospheric models and katabatic winds, in ventilation design, dense gas dispersion, insulation with double pane window, and solar collectors \cite{Songul}, to just name a few examples. We consider the Boussinesq system on a finite, connected domain $\Omega\subset \R^d$ ($d=2,3$) with  boundary $\Gamma_1\cup\Gamma_2$. We denote $u$ to be the fluid velocity, $p$ is the pressure, $T$ is the temperature (or density), $\nu>0$ is the kinematic viscosity of the fluid, $\kappa>0$ the thermal diffusivity, $f$ the external forcing term, and $R_i>0$ the Richardson number which accounts for the gravitational force and thermal expansion of the fluid. This system is given by
\begin{equation}\label{Bouss}
\begin{cases}
u\cdot \nabla u +\nabla p -\nu\Delta u &= f +R_i  (0 ~T)^T \text{ in }\Omega,\\
\nabla\cdot u&=0 \text{ in }\Omega,\\
u\cdot \nabla T-\kappa \Delta T&=g \text{ in }\Omega,\\
\end{cases}
\end{equation}
with no-slip boundary conditions for velocity and mixed homogenous Dirichlet and homogenous Neumann for temperature (the latter of which corresponds to perfect insulation):
\begin{equation}\label{BoussBC}
\begin{cases}
u&= 0 \text{ on }\partial\Omega,\\
T&=0 \text{ on }\partial\Gamma_1,\\
\nabla T\cdot n&=0 \text{ on }\partial\Gamma_2.
\end{cases}
\end{equation}

The physical constants that the model relies on are expressed as the Reynolds number $Re$, the Prandtl number $Pr$, and the Rayleigh number $Ra$, which are given by
\begin{equation*}
Re=\nu^{-1},~~~~Pr=\frac{\nu}{\kappa}, ~~~~ Ra=Ri ~ Re^2 ~ Pr.\\
\end{equation*}
For any set of problem data ($\nu, ~Ri, ~\kappa,~f,~g$), the system \eqref{Bouss}-\eqref{BoussBC} is known to admit weak solutions \cite{Layton89}. Under a smallness condition on this data, the system is known to be well-posed; for sufficiently high $Ra$, this system may lose uniqueness \cite{Layton89}.\\

The Picard iteration is a common method for solving the system \eqref{Bouss}-\eqref{BoussBC}, and is given by
\begin{equation}\label{eq:Picard}
\begin{cases}
\uk\cdot \nabla \ukone+\nabla p^{k+1} -\nu\Delta \ukone &= f +R_i  (0 ~\tkone)^T ,\\
\nabla\cdot \ukone&=0 ,\\
\uk\cdot \nabla \tkone-\kappa \Delta \tkone&=g,
\end{cases}
\end{equation}
together with \eqref{BoussBC} for $\ukone$ and $\tkone$.
Note that the system equations \eqref{eq:Picard} are decoupled because the velocity term in the heat transport equation is known.  This makes each iteration of Picard efficient since one can first solve for $\tkone$ and then solve an Oseen Problem for $\ukone$ and $p^{k+1}$. Picard admits stable solutions for any problem data and produces unique solutions provided the problem data satisfies a smallness condition on the problem data (see Section \ref{Prelim}). However, this iteration converges linearly for sufficiently small problem data (see Section \ref{Prelim}) but the linear convergence rate decreases as $Ra$ increases and fails for even moderate $Ra$ \cite{transport,NC1}. \\

The purpose of this paper is to improve Picard for the Boussinesq equations in the setting of where partial solution data is available, by incorporating continuous data assimilation (CDA) style nudging into the Picard iteration (a method we will call CDA-Picard). Partial solution data may be available from physical experiments, measurements of physical phenomena, or even from a high resolution simulation where passing all solution data is too expensive. Note that the collected partial solution data could include noise, such as that which occurs from physical measurements. \\

CDA ideas have recently been found successful for solving steady Navier-Stokes equations (NSE) using Picard \cite{ALNR24}, which is perhaps surprising since CDA is designed for continuous in time assimilation. We consider herein a natural extension of this idea to multiphysics problems. Analysis of this type of CDA applied to Boussinesq poses significant extra difficulties and considerations compared to NSE. This is because the Boussinesq equations are given by NSE coupled to a heat transport equation, which includes extra nonlinear terms. Furthermore, unlike NSE, Boussinesq has the unknown variable $T$ as well as $u$ thus CDA nudging is applicable to multiple variables. Hence, there are extra cases to consider for CDA applied to the Boussinesq equations: nudging both $u$ and $T$, nudging just $u$, and nudging just $T$ (spoiler: it is enough to nudge just $u$, but better to nudge both). \\

CDA is performed using an interpolant $I_H$ where $H$ is representative of the spacing of collected data and $I_H(u)$ and $I_H(T)$ are interpolants of the partial solution data projected in the solution space. Thus, CDA-Picard for the Boussinesq equations takes the form (together with \eqref{BoussBC})
\begin{equation}\label{CDA}
\begin{cases}
\uk\cdot \nabla \ukone+\nabla p^{k+1} -\nu\Delta \ukone + \mu_1I_H(\ukone-u)&= f +R_i  (0 ~\tkone)^T, \\
\nabla\cdot \ukone&=0 ,\\
\uk\cdot \nabla \tkone-\kappa \Delta \tkone+\mu_2I_H(\tkone-T)&=g ,
\end{cases}
\end{equation}
where $\mu_1>0$ and $\mu_2>0$ are user chosen nudging constants. For simplicity we restrict \eqref{CDA} to the case where $I_H(u)$ and $I_H(T)$ are data collected at the same locations. Because $I_H(u)$ and $I_H(T)$ may come from collected data it is a natural assumption that the data includes noise (inaccuracies). This will inevitably affect the convergence and accuracy of CDA-Picard and therefore we investigate this case as well.\\

CDA was first proposed by Azouani, Olson, and Titi in 2014 \cite{AOT14} for time dependent systems and has since been applied to a wide variety of time dependent problems including NSE and turbulence \cite{AOT14,DMB20,FLT19,CGJP22}, the Cahn-Hilliard equation \cite{DR22}, planetary geostrophic modeling \cite{FLT16}, Benard convection \cite{FJT15}, and many others. Interest in CDA has increased in the last decade leading to many improvements to the algorithm and uses for it, such as for sensitivity analyses \cite{CL21}, parameter recovery \cite{CHL20,CH22}, numerical methods and analyses \cite{IMT20,LRZ19,RZ21,DR22,GNT18,JP23}. Of recent interest is modification to the algorithm such as adaptive nudging \cite{LF2025} and direct enforcement \cite{RebXue25}. There have also been several works investigating the use of the algorithm with restrictions on the observational data, including measurement error \cite{BOT15, CO23,FMT16} and partial observations of the problem variables \cite{FJT15, FLT16_ab, FLT16_ben, FLT16_cha}.

We prove herein that in the absence of noise in the partial solution data, CDA-Picard has a convergence rate that is improved by a factor of $H^{1/2}$ compared to Picard, when nudging both $u$ and $T$. We also prove that when nudging just $u$ there is a similar improvement, but not when nudging just $T$. We also prove that CDA-Picard using noisy solution data similarly scales the convergence rate of the residual by $H^{1/2}$ and concurrently the limit solution accuracy is bounded by the accuracy of the solution data. We find the limit solution of CDA-Picard with noisy data is typically within the convergence basin of Newton, meaning that a modified method incorporating Newton can overcome the accuracy limitations caused by the noisy data.\\

This paper is arranged as follows. First we present notation, and then stability and convergence results for Picard in Section \ref{Prelim}. These Picard results are known, and will be important for comparison. Then we give analysis in Section \ref{Analysis} for CDA-Picard when nudging both $u$ and $T$, just $u$, and just $T$. Section \ref{noisyAnalysis} considers these same questions as Section \ref{Analysis} but with noisy data. In Section \ref{Numerical} we use CDA-Picard to solve a benchmark problem both with and without noise, and propose a modification of CDA-Picard with noisy data.

\section{Preliminaries}\label{Prelim}

Let $\Omega\subset\R^d$ ($d=2\text{ or }3$) be a connected domain with boundary $\partial\Omega=\Gamma_1\cup\Gamma_2$ satisfying $meas(\Gamma_1\cap\Gamma_2)=0$.  Let $(\cdot,\cdot)$ and $\|\cdot \|$ denote the $L^2(\Omega)$ inner product and $L^2(\Omega)$ norm, respectively. We define the pressure, temperature, velocity and divergence free velocity solution spaces as 
\begin{align*}
Q&=\left\{q\in L^2(\Omega): \int_\Omega q ~d\Omega=0\right\},\\
 D&=\left\{S\in H^1(\Omega): S|_{\Gamma_1}=0\right\},\\
X&=\left\{v\in H^1:  v|_{\partial\Omega}=0 \right\},\\
V&=\left\{v\in X: (\nabla\cdot v,q)=0 ~\forall q\in Q\right\}.
\end{align*}

Let $X^\ast$, $V^\ast$, and $D^\ast$ represent the dual spaces of $X$, $V$, and $D$, respectively. We use $(\cdot,\cdot)$ to also denote the dual pairings of these spaces.
  For $z\in X, ~V, \text{ or } D$,  it is known that the Poincar\'e inequality holds:
  \begin{equation*}
  \|z\| \leq C_p \|\nabla z\|,
  \end{equation*}
  where $C_p>0$ is a constant depending only on $\Omega$.\\
  
 Let $b:X\times X\times X\rightarrow \R$ and $\hat{b}: X\times D\times D\rightarrow \R$ denote the trilinear functionals given by:
 \begin{align*}
 b(u,v,w)&:=(u\cdot \nabla v,w)+\frac{1}{2}((\nabla\cdot u)v,w),\\
 \hat{b}(u,v,w)&:=(u\cdot \nabla s,t)+\frac{1}{2}((\nabla\cdot u)s,t).
 \end{align*}
 These trilinear functionals defined above are skew-symmetric,
 $$b(u,v,v)= 0 \text{ and } \hat{b}(u,s,t)=0.$$
 We will bound $b$ and $\hat{b}$ herein using the well known bounds \cite{Laytonbook}:
 $\forall u,v,w \in X$, $s,t\in D$, $\exists C_s>0$ depending only on $|\Omega|$ such that
\begin{align}\label{bbound2}
|b(u,v,w)|&\leq C_s \|\nabla u\| \|\nabla v\| \|w\|^{1/2} \|\nabla w\|^{1/2},\\
|\hat{b}(u,s,t)|&\leq C_s \|\nabla u\| \|\nabla s\| \|t\|^{1/2}\|\nabla t\|^{1/2},
\end{align}
and 
\begin{align}\label{bbound}
|b(u,v,w)|&\leq C_p^{1/2} C_s \|\nabla u\| \|\nabla v\| \|\nabla w\|,\\
 |\hat{b}(u,s,t)|&\leq C_p^{1/2} C_s \|\nabla u\| \|\nabla s\| \|\nabla t\|.
\end{align}

\subsection{Boussinesq preliminaries}
The system \eqref{Bouss}-\eqref{BoussBC} is known to admit stable solutions for any $Ra>0$ \cite{Layton89}, which can be proven similarly to the analogous result for the steady NSE by using the Leray-Schauder theorem. However, uniqueness for \eqref{BoussWF} requires a smallness condition on the data. These results are provided below and we give the proofs for the following results in the appendix for completeness.\\

The Galerkin weak formulation of the Boussinesq equations takes the form: Given $f\in X^\ast$ and $g\in D^\ast$, find $(u,p,T)\in X\times Q\times D$ satisfying $\forall (v,q,w)\in (X\times Q\times D)$
\begin{equation*}
\begin{cases}
b(u , u ,v)-( p,\nabla\cdot v) +\nu(\nabla u,\nabla v) &= (f,v) +(R_i (0~T)^T,v),\\
(\nabla\cdot u,q)&=0,\\
\hat{b}(u , T,w)+\kappa (\nabla T,\nabla w)&=(g,w).
\end{cases}
\end{equation*}
Note that the perfect insulation condition $\nabla T\cdot n|_{\Gamma_2}$ is weakly enforced. The spaces $X, Q, \text{ and }D$ satisfy an LBB condition. Therefore the problem can be reformulated using $V$ as:  Given $f\in V^\ast$ and $g\in D^\ast$, find  $(u,T)\in V\times D$ satisfying for all $(v,w)\in V\times D$
\begin{equation}\label{BoussWF}
\begin{cases}
b(u , u ,v)+\nu(\nabla u,\nabla v) &= (f,v) +(R_i (0~T)^T,v),\\
\hat{b}(u , T,w)+\kappa (\nabla T,\nabla w)&=(g,w).
\end{cases}
\end{equation}

\begin{lem}\label{BoussStability}
Any solution to the Boussinesq equations \eqref{BoussWF} satisfies the a priori estimate
\begin{align}
\|\nabla T\|&\leq \|\nabla T\| \leq \kappa^{-1}\|g\|_{D^\ast}=: M_2, \label{TBound}\\
\|\nabla u\|&\leq \|\nabla u\| \leq \nu^{-1}\| f\|_{V^\ast} +R_i C_p^2\nu^{-1} M_2=: M_1\label{uBound}.
\end{align}
\end{lem}

\begin{lem}\label{lem:BoussUnique}
Let $\alpha_1=C_s \nu^{-1} M_1$ and $\alpha_2=C_s \kappa^{-1} M_2$. If $C_p^2 \nu^{-1}R_i , ~C_p^{1/2}(\alpha_2 + \alpha_1)<1$, then the solutions to \eqref{BoussWF} are unique.
\end{lem}

\subsection{The Picard iteration for the Boussinesq equations}\label{PicardSection}
We provide results for Picard applied to the Boussinesq equation \eqref{eq:Picard} and include the proofs in the appendix for completeness. These results are similar to those found in \cite{PNBouss,Laytonbook}.\\

 The weak formulations for Picard for the Boussinesq system takes the form: Find $(\ukone,\tkone)\in V\times D$ satisfying $
\forall (v,w)\in V\times D$,
\begin{equation}\label{PicardWF}
\begin{cases}
b(\uk, \ukone,v)  +\nu(\nabla \ukone,\nabla v) &= (f,v) +R_i ((0 ~\tkone)^T,v),\\
\hat{b}(\uk, \tkone,w)+\kappa (\nabla \tkone,\nabla w)&=(g,w).
\end{cases}
\end{equation}
Note that the Picard iteration decouples the temperature equation and thus solving \eqref{PicardWF} is a two-step process where one first solves a scalar convection-diffusion problem and then an Oseen problem. Effective preconditioners for these linear systems exist in the literature \cite{benzi,elman:silvester:wathen,HR13,BB12}.

\begin{lem}\label{Lemma:PicardBound}
Any solution to the Picard iteration for the Boussinesq equations satisfies the a priori estimate: for any $k=1,2...$,
\begin{align*}
 \|\nabla T^k\|\leq M_2,\\
\|\nabla u^k\| \leq M_1,
\end{align*}
\end{lem}

\begin{lem}
The Picard iteration \eqref{PicardWF} with data satisfying $\min\left\{1-\nu^{-1}\frac{C_p^2 R_i}{2} ,1-\kappa^{-1}\frac{C_p^2 R_i}{2} \right\}>0$, admits a unique solution. 
\end{lem}

\begin{lem}\label{Lemma:PicardConv}
Consider the Picard iteration \eqref{PicardWF} with data satisfying $C_p^2\nu^{-1}R_i<1$ and $C_p^{1/2}(\alpha_1+\alpha_2) <1$. Then the iteration converges linearly with rate $C_p^{1/2}(\alpha_1+\alpha_2)$. In particular we have
\begin{equation*}
 \|\nabla (T-T^{k+1}) \| \leq  C_p^{1/2}\alpha_2 \|\nabla(u-u^k)\|,
\end{equation*}
and
\begin{equation*}
\|\nabla (u-u^{k+1})\| \leq C_p^{1/2}(\alpha_1+ \alpha_2) \|\nabla(u-u^k)\|.
\end{equation*}
\end{lem}

\begin{rem}
Recall that Lemma \ref{lem:BoussUnique} gives the sufficient condition for uniqueness of Boussinesq solutions $C_p^2 \nu^{-1}R_i, C_p^{1/2}(\alpha_1+\alpha_2)<1$. This means that the conditions for convergence of Picard and the uniqueness of solutions to the Boussinesq equations are the same.
\end{rem}

\section{Convergence of CDA-Picard for Boussinesq equations with (accurate) partial solution data}\label{Analysis}
We now provide analytical results for CDA-Picard that uses accurate partial solution data. We begin by providing results for CDA-Picard when nudging both $u$ and $T$ and for nudging only $u$ or $T$. To prove convergence, we will use the weighted $H^1$ norm 
$$
\|v\|_{\ast}:=\sqrt{\frac{1}{ 4 C_I^2 H^2}\|v\|^2 + \|\nabla v\|^2}.
$$
This weighted $H^1$ norm naturally arises in the CDA-Picard analysis. For small $H$, the $\ast$-norm is essentially equivalent to the $L^2(\Omega)$ norm. Following \cite{ALNR24}, a modified analysis with additional assumptions could allow for $L^2$ norm estimates instead.

Let $\tau_H(\Omega)$ denote a coarse mesh on $\Omega$ with maximum element length $H$, and let $\mathbb{P}_k(\tau_H)$ denote the degree $k$ continuous Lagrange finite element defined on $\tau_H$. Then we define the finite element spaces $X_H=P_k(\tau_H)\cap X$ and $D_H = P_k(\tau_H)\cap D$. We define the interpolation operator $I_H$  on $\tau_H(\Omega)$. We assume that $I_H$ satisfies: $\exists$ a constant $C_I$ independent of $H$ satisfying
\begin{align}
\|I_Hv-v\|\leq C_I H\|\nabla v\| \hspace{5pt} &\forall v\in X \text{ and } D,\\
\|I_Hv\|\leq C_I\| v\| \hspace{5pt} &\forall v\in X \text{ and } \forall v\in X \text{ and } D.
\end{align}
In our numerical tests, we take $I_H$ to be the $L^2$ projection onto $X_H=P_0(\tau_H)$.\\

We have the following bounds on the $L^2(\Omega)$ norm.
\begin{lem}\label{starbound}
For any $v\in X,D$ all of the following hold:
\begin{align*}
\|v\|^{1/2}\|\nabla v\|^{1/2}& \leq (C_I H)^{1/2}\|v\|_{\ast},\\
 \|v\|& \leq C_p \|v\|_{\ast},\\
  \|v\|&\leq C_p^{1/2}(C_I H)^{1/2} \|v\|_{\ast}.
\end{align*}
\end{lem}
\begin{proof}
Using Young's inequality and the definition of the $\ast$-norm we obtain
\begin{equation}\label{eq:SB1}
\|v\|^{1/2}\|\nabla v\|^{1/2}= \left(\frac{1}{4C_I H} \|v\|^2 + C_I H\|\nabla v\|^2\right)^{1/2} \leq (C_I H)^{1/2}\|v\|_{\ast}.
\end{equation}
Using \eqref{eq:SB1} and the Poincare inequality we get
\begin{align*}
 \|v\|&\leq C_p\|\nabla v\| \leq C_p \|v\|_{\ast},\\
  \|v\|&= C_p^{1/2}\|v\|^{1/2}\|\nabla v\|^{1/2}\leq C_p^{1/2}(C_I H)^{1/2} \|v\|_{\ast}.
\end{align*}
\end{proof}

We first assume that exact solution data  is known at some set of points in $\Omega$.
Then the weak formulation of CDA-Picard for the Boussinesq equations \eqref{CDA} is given by: Given $f\in X^\ast$ and $g^\ast\in D^\ast$, find $(u,T)\in X\times D$ satisfying
\begin{equation}\label{CDAWF}
\begin{cases}
b(\uk, \ukone ,v)+\nu(\nabla \ukone,\nabla v) +\mu_1(I_H(\ukone-u),I_H(v)) &= (f,v) +(R_i (0~\tkone)^T,v)\\
\hat{b}(\uk, \tkone,w)+\kappa (\nabla \tkone,\nabla w)+\mu_2(I_H(\tkone-T),I_H(w))&=(g,w).
\end{cases}
\end{equation}
Note that the term $\mu_1(I_H(\ukone-u),I_H(v))$ does not follow the usual rules for the Galerkin weak formulation and may be considered a variational crime, if $I_H$ is not the $L^2$ projection onto $X_H$.

\subsection{Convergence of CDA-Picard nudging both $u$ and $T$}

 We begin analyzing CDA-Picard \eqref{CDAWF} without noise when nudging both $u$ and $T$, and then consider the case of nudging just $u$ and just $T$ which use similar techniques.

\begin{theorem}\label{thm:nonoise}
Let $\mu_1\geq \frac{\nu}{4C_I^2 H^2}$, $\mu_2\geq \frac{\kappa}{4C_I^2H^2}$, $R_i \nu^{-1}C_p^2<1$, $C_p^{-1/2} \frac{4}{3}(C_I H)^{1/2}<1$ and $\frac{4}{3}\sqrt{C_IH}(\alpha_1+\alpha_2)<1$. Then CDA-Picard converges linearly with rate at least $ \frac{4}{3}\sqrt{C_IH}(\alpha_1+\alpha_2)$:
$$
\|u-u^{k+1}\|_\ast+(1- \frac{4}{3}R_i \nu^{-1}C_p^2 )\| T-T^{k+1}\|_{\ast}\leq \frac{4}{3}(\alpha_1+\alpha_2)(C_I H)^{1/2}\|u-u^k\|_\ast . 
$$
\end{theorem}

\begin{rem}
The condition $R_i \nu^{-1}C_p^2<1$ is the condition for well-posedness of the Boussinesq equations from Lemma \ref{lem:BoussUnique}. In the analysis, it appears together with the restriction $C_p^{-1}\frac{4}{3}\sqrt{C_I H}\leq 1$. The analysis can be changed slightly to give the restriction
$\frac{4}{3}R_i \nu^{-1}C_p^{3/2}<1$
or the restriction $R_i \nu^{-1}C_p^{3/2}<1\text{ and } \sqrt{C_I H}\leq \frac{3}{4}$.
\end{rem}

\begin{rem}
Without CDA, Lemma \ref{Lemma:PicardConv} shows a sufficient condition for Picard convergence is $C_p^{1/2}(\alpha_1+\alpha_2)<1$, and the convergence rate is $C_p^{1/2}(\alpha_1+\alpha_2)$. Theorem \ref{thm:nonoise} thus shows how CDA improves convergence rate and even enables convergence when Picard fails: the right hans sidde is scaled by $H^{1/2}$.
\end{rem}

\begin{proof}
Let $\ekone=u-u^{k+1}$ and $\etkone=T-T^{k+1}$. We subtract \eqref{BoussWF} from \eqref{CDAWF} and set $v=\ekone$ and $w=\etkone$. Then we add and subtract terms to get
$$\begin{cases}
b(\ek,\ekone,\ekone)+b(\ek, u,\ekone)+b(u,\ekone,\ekone)+\nu\|\nabla \ekone\|^2+\mu_1\|I_H(\ekone)\|^2 &= (R_i (0 ~\etkone)^T,\ekone),\\
\hat{b}(u,\etkone,\etkone)+\hat{b}(\ek, \etkone,\etkone)+\hat{b}(\ek , T,\etkone)+\kappa \|\nabla \etkone\|^2+\mu_2\|I_H(\etkone)\|^2&=0.
\end{cases}$$
The skew symmetry of $b$ and $\hat{b}$ vanishes four nonlinear terms yielding
$$\begin{cases}
\nu\|\nabla \ekone\|^2 +\mu_1\|I_H(\ekone)\|^2&= (R_i (0 ~\etkone)^T,\ekone)-b(\ek, u,\ekone),\\
\kappa \|\nabla \etkone\|^2+ \mu_2\|I_H(\etkone)\|^2&=-\hat{b}(\ek,  T,\etkone).
\end{cases}$$
We next use Lemma \ref{starbound}, \eqref{bbound2}, and Lemma \ref{BoussStability} to upper bound the right hand side as
\begin{equation}\label{CDA1}
\begin{cases}
\nu\|\nabla \ekone\|^2 +\mu_1\|I_{H}(\ekone)\|^2&\leq R_i C_p^{3/2} (C_IH)^{1/2}\| \etkone\|_\ast \|  \ekone\|_{\ast} +\alpha_1\nu  ( C_I H)^{1/2}\|\ek\|_\ast \|\ekone\|_\ast,\\
\kappa \|\nabla \etkone\|^2 + \mu_2\|I_{H}(\etkone)\|^2 &\leq \alpha_2\kappa  ( C_I H)^{1/2}\|\ek\|_\ast \|\etkone\|_\ast.
\end{cases}
\end{equation}

Next we lower bound the left hand sides using the interpolation bound and the triangle inequality via
$$\begin{cases}
\nu\|\nabla \ekone\|^2 +\mu_1\|I_{H}(\ekone)\|^2&= \frac{\nu}{4}\|\nabla \ekone\|^2 +\mu_1\|I_{H}(\ekone)\|^2 + \frac{3\nu}{4}\|\nabla \ekone\|^2,\\
&\geq  \frac{\nu}{4 C_I^2 H^2}\|\ekone -I_{H}(\ekone)\|^2 +\mu_1\|I_{H}(\ekone)\|^2 + \frac{3\nu}{4}\|\nabla \ekone\|^2,\\
&\geq  \frac{\nu}{4 C_I^2 H^2}(\|\ekone -I_{H}(\ekone)\|^2 +\|I_{H}(\ekone)\|^2) + \frac{3\nu}{4}\|\nabla \ekone\|^2,\\
&\geq \frac{\nu}{4 C_I^2 H^2}\|\ekone \|^2 + \frac{3\nu}{4}\|\nabla \ekone\|^2,\\
&\geq  \frac{3\nu}{4}\| \ekone\|^2_\ast,\\
\kappa \|\nabla \etkone\|^2 + \mu_2\|I_{H}(\etkone)\|^2&= \frac{\kappa}{4} \|\nabla \etkone\|^2 + \mu_2\|I_{H}(\etkone)\|^2 + \frac{3\kappa}{4} \|\nabla \etkone\|^2,\\
 &\geq \frac{\kappa}{4C_I^2H^2} \| \etkone-I_{H}(\etkone)\|^2 + \mu_2\|I_{H}(\etkone)\|^2 + \frac{3\kappa}{4} \|\nabla \etkone\|^2,\\
  &\geq \frac{\kappa}{4 C_I^2 H^2} \| \etkone\|^2 + \frac{3\kappa}{4} \|\nabla \etkone\|^2,\\
 &\geq \frac{3\kappa}{4} \| \etkone\|^2_\ast,\\
\end{cases}$$
 by assumptions on $\mu_i$ being sufficiently large. Using this in \eqref{CDA1} and simplifying yields 

$$\begin{cases}
\| \ekone\|_\ast &\leq  \frac{4}{3} \nu^{-1} R_i C_p^{3/2} (C_I H)^{1/2}\| \etkone\|_\ast+\frac{4}{3}\alpha_1C_p( C_I H)^{1/2}\|\ek\|_\ast,\\
 \| \etkone\|_\ast&\leq  \frac{4}{3}\alpha_2 C_p ( C_I H)^{1/2}\|\ek\|_\ast.
\end{cases}$$

Finally adding the equations gives the desired result.
\end{proof}

\begin{theorem}\label{thm:nonoiseresidual}
Let $\mu_1\geq \frac{\nu}{4C_I^2 H^2}$, $\mu_2\geq \frac{\kappa}{4C_I^2H^2}$, $R_i \nu^{-1}C_p^2<1$, $C_p^{-1/2} \frac{4}{3}(C_I H)^{1/2}<1$ and $\frac{4}{3}\sqrt{C_IH}(\alpha_1+\alpha_2)<1$. Then the residual of CDA-Picard converges linearly with rate at least $ \frac{4}{3}\sqrt{C_IH}(\alpha_1+\alpha_2)$:
$$
\|u^{k+1}-u^k\|_\ast+(1- \frac{4}{3}R_i \nu^{-1}C_p^2 )\| T^{k+1}-T^k\|_{\ast}\leq \frac{4}{3}(\alpha_1+\alpha_2)(C_I H)^{1/2}\|u^k-u^{k-1}\|_\ast . 
$$
\end{theorem}

\begin{rem}
The conditions and rates for convergence of the error and residual for CDA-Picard are equivalent.
\end{rem}

\begin{proof}
Let $\ekone=u^k-u^{k+1}$ and $\etkone=T^k-T^{k+1}$. We subtract \eqref{CDAnoiseWF} at time $k$ and time $k+1$, add and subtract terms, and set $v=\ekone$ and $w=\etkone$ to get
\begin{equation}\label{eq:res1}
\begin{cases}
b(e^k , \uk ,\ekone)+b(\uk , \ekone ,\ekone)+\nu\|\nabla \ekone\|^2 +\mu_1\| I_H(\ekone) \|^2 &= (R_i (0~\etkone)^T,\ekone),\\
\hat{b}(e^k, T^k,\etkone)+\hat{b}(\uk, \etkone ,\etkone)+\kappa\|\nabla \etkone\|^2+\mu_2 \|I_H(\etkone)\|^2&=0.
\end{cases}
\end{equation}
The skew symmetry of $b$ and $\hat{b}$ vanishes nonlinear terms in \eqref{eq:res1} yielding,
$$
\begin{cases}
b(e^k , \uk ,\ekone)+\nu\|\nabla \ekone\|^2 +\mu_1\| I_H(\ekone) \|^2 &= (R_i (0~\etkone)^T,\ekone),\\
\hat{b}(e^k, T^k,\etkone)+\kappa\|\nabla \etkone\|^2+\mu_2 \|I_H(\etkone)\|^2&=0.
\end{cases}
$$
The solutions $\uk$, $T^k$ for $k=1,...$ satisfy the same stability bound as Lemma \ref{Lemma:PicardBound}. We use this and \eqref{bbound2} to obtain
$$
\begin{cases}
\nu\|\nabla \ekone\|^2 +\mu_1\| I_H(\ekone) \|^2 &\leq R_iC_p^{3/2}(C_I H)^{1/2} \|\etkone\|_\ast \|\ekone\|_\ast + \alpha_1\nu(C_I H)^{1/2}\| e^k\|_\ast \|\ekone\|_\ast,\\
\kappa\|\nabla \etkone\|^2+\mu_2 \|I_H(\etkone)\|^2&\leq \alpha_2\kappa (C_I H)^{1/2}\| e^k\|_\ast \|\etkone\|_{\ast}.
\end{cases}
$$
We lower bound the left hand side analogously to Theorem \ref{thm:nonoise},
$$
\begin{cases}
\|\ekone\|_\ast &\leq \frac{4}{3}\nu^{-1}R_i C_p^{3/2}(C_I H)^{1/2} \|\etkone\|_\ast + \frac{4}{3} \alpha_1(C_I H)^{1/2}\|e^k\|_\ast, \\
\|\etkone\|_\ast&\leq \frac{4}{3} \alpha_2 (C_I H)^{1/2}\|e^k\|_\ast .
\end{cases}
$$
Finally we add the equations and simplify to get the desired results.
\end{proof}


If we nudge just $u$ then this is equivalent to setting $\mu_2=0$ in \eqref{CDAWF} which gives the following result.
\begin{cor}
Let $\mu_1\geq \frac{\nu}{4C_I^2 H^2}$, $\mu_2=0$, $R_i \nu^{-1}C_p^2<1$, $C_p^{-1/2} \frac{4}{3}(C_I H)^{1/2}<1$ and $(C_I H)^{1/2}(\frac{4}{3}\alpha_1+C_p^{1/2}\alpha_2)<1$. Then  CDA-Picard converges linearly with rate at least $ (C_I H)^{1/2}(\frac{4}{3}\alpha_1+C_p^{1/2}\alpha_2)$:
$$
\| u-u^{k+1}\|_\ast +(1- \frac{4}{3} \nu^{-1} R_i C_p^{3/2} ) \| \nabla (T-T^{k+1})\| \leq \frac{4}{3}( C_I H)^{1/2}(\alpha_2+\alpha_1)\|u-u^k\|_\ast,
$$
\end{cor}

\begin{proof}
Let $\ekone=u-u^{k+1}$ and $\etkone=T-T^{k+1}$. We subtract \eqref{BoussWF} from \eqref{CDAWF} and set $v=\ekone$ and $w=\etkone$. Analagously to Theorem \ref{thm:nonoise} we add and subtract terms and using
the skew symmetry of $b$ and $\hat{b}$ vanishes four nonlinear terms yielding
$$\begin{cases}
\nu\|\nabla \ekone\|^2 +\mu_1\|I_H(\ekone)\|^2&= (R_i (0 ~\etkone)^T,\ekone)-b(\ek, u,\ekone),\\
\kappa \|\nabla \etkone\|^2&=-\hat{b}(\ek,  T,\etkone).
\end{cases}$$
We use Lemma \ref{starbound}, \eqref{bbound2}, Lemma \ref{BoussStability} to upper bound the right hand side as
\begin{equation}\label{CDA2}
\begin{cases}
\nu\|\nabla \ekone\|^2 +\mu_1\|I_{H}(\ekone)\|^2&\leq R_i C_p^{3/2}(C_IH)^{1/2}\| \nabla \etkone\| \|  \ekone\|_{\ast} +\alpha_1\nu( C_I H)^{1/2}\|\ek\|_\ast \|\ekone\|_\ast,\\
\kappa \|\nabla \etkone\|^2 &\leq \alpha_2 C_p^{1/2}(C_I H)^{1/2} \kappa\|\ek\|_\ast \|\nabla \etkone\|.
\end{cases}
\end{equation}

Next we lower bound the left hand sides using the interpolation bound and the triangle inequality as in Theorem \ref{thm:nonoise} 
 and use this in \eqref{CDA2} to get
$$\begin{cases}
\| \ekone\|_\ast &\leq  \frac{4}{3} \nu^{-1} R_i C_p^{3/2}(C_IH)^{1/2} \| \nabla \etkone\|+\frac{4}{3}\alpha_1( C_I H)^{1/2}\|\ek\|_\ast,\\
 \| \nabla \etkone\|&\leq  \alpha_2 C_p^{1/2} (C_I H)^{1/2}\|\ek\|_\ast  
\end{cases}$$

Finally adding the equations gives the desired result.
\end{proof}

If we nudge just $T$ then this is equivalent to setting $\mu_1=0$ in \eqref{CDAWF} which only gives the following result.
\begin{cor}
Let $\mu_1=0$, $\mu_2\geq \frac{\nu}{4C_I^2 H^2}$, $R_i \nu^{-1}C_p^2<1$, and $(C_p^{1/2}\alpha_1+\frac{4}{3}\sqrt{C_IH}\alpha_2)<1$. Then  CDA-Picard converges linearly with rate at least $ (C_p^{1/2}\alpha_1+\frac{4}{3}\sqrt{C_IH}\alpha_2)$:
$$
\| \nabla (u-u^{k+1})\| + (1-\nu^{-1} R_i C_p^{2} )\| T-T^{k+1}\|_\ast  \leq (C_p^{1/2}\alpha_1+ \frac{4}{3}\alpha_2  ( C_I H)^{1/2})\|\nabla (u-u^k)\|  $$
\end{cor}

\begin{proof}
Let $\ekone=u-u^{k+1}$ and $\etkone=T-T^{k+1}$. We subtract \eqref{BoussWF} from \eqref{CDAWF} and set $v=\ekone$ and $w=\etkone$. Analagously to Theorem \ref{thm:nonoise} we add and subtract terms and using 
 the skew symmetry of $b$ and $\hat{b}$ vanishes 4 terms yielding
$$\begin{cases}
\nu\|\nabla \ekone\|^2 &= (R_i (0 ~\etkone)^T,\ekone)-b(\ek, u,\ekone),\\
\kappa \|\nabla \etkone\|^2+ \mu_2\|I_H(\etkone)\|^2&=-\hat{b}(\ek,  T,\etkone).
\end{cases}$$
We use Lemma \ref{starbound}, \eqref{bbound2}, \eqref{bbound}, Lemma \ref{BoussStability} to upper bound the right hand side as
\begin{equation}\label{CDA3}
\begin{cases}
\nu\|\nabla \ekone\|^2 +\mu_1\|I_{H}(\ekone)\|^2&\leq R_i C_p^{2}\| \etkone\|_\ast \| \nabla \ekone\| +\alpha_1\nu C_p^{1/2}\|\nabla\ek\| \|\nabla \ekone\|,\\
\kappa \|\nabla \etkone\|^2 + \mu_2\|I_{H}(\etkone)\|^2 &\leq \alpha_2\kappa  ( C_I H)^{1/2}\|\nabla \ek\| \|\etkone\|_\ast.
\end{cases}
\end{equation}

Next we lower bound the left hand sides using the interpolation bound and the triangle inequality analogously to Theorem \ref{thm:nonoise} 
 and use this in \eqref{CDA3} to get
$$\begin{cases}
\| \nabla \ekone\| &\leq R_i C_p^{2} \| \etkone\|_\ast+\alpha_1C_p^{1/2}\|\nabla \ek\|,\\
 \| \etkone\|_\ast&\leq  \frac{4}{3}\alpha_2  ( C_I H)^{1/2}\|\nabla \ek\|.
\end{cases}$$

Finally adding the equations gives the desired result.
\end{proof}

Therefore, with accurate solution data there is a direct correlation between convergence, accuracy, and the amount of data collected $H$. Furthermore, CDA does not worsen the restrictions of Picard, so it can only improve convergence properties. 

\section{CDA-Picard for Boussinesq equation with noisy data}\label{noisyAnalysis}
We suppose now that the data has noise (error) given by some $\epsilon(x)$, meaning that we have data given by $u+\epsilon_u$ and $T+\epsilon_T$ on the data points. Then the weak formulation of CDA-Picard for the Boussinesq equations \eqref{CDA} with noise is given by: Given $f\in X^\ast$ and $g^\ast\in D^\ast$, find $(u,T)\in X\times D$ satisfying
\begin{equation}\label{CDAnoiseWF}
\begin{cases}
b(\uk, \ukone ,v)+\nu(\nabla \ukone,\nabla v) +\mu_1(I_H(\ukone-(u+\epsilon_u) ),I_H(v)) &= (f,v) +(R_i (0~\tkone)^T,v)\\
\hat{b}(\uk, \tkone,w)+\kappa (\nabla \tkone,\nabla w)+\mu_2(I_H(\tkone-(T+\epsilon_T)),I_H(w))&=(g,w).
\end{cases}
\end{equation}

\subsection{Convergence of CDA-Picard for Boussinesq equation with noisy data}
We now we analyze CDA-Picard with noise \eqref{CDAnoiseWF}  when nudging both $u$ and $T$. We begin by analyzing the convergence of the residual and then the error.

\begin{theorem}
Let $\mu_1\geq \frac{\nu}{4C_I^2 H^2}$, $\mu_2\geq \frac{\kappa}{4C_I^2H^2}$, $R_i \nu^{-1}C_p^2<1$, $C_p^{-1/2} \frac{4}{3}(C_I H)^{1/2}<1$ and $\frac{4}{3}\sqrt{C_IH}(\alpha_1+\alpha_2)<1$. Then residual of CDA-Picard converges linearly with rate at least $ \frac{4}{3}\sqrt{C_IH}(\alpha_1+\alpha_2)$:
$$\|u^k-u^{k+1}\|_\ast+(1-\frac{4}{3}\nu^{-1}R_i C_p^{3/2}(C_I H)^{1/2} )\|T^k-T^{k+1}\|_\ast \leq \frac{4}{3} (C_I H)^{1/2}(\alpha_1+  \alpha_2)\|u^{k-1}-u^k\|_\ast .$$
\end{theorem}

\begin{rem}
The sufficient conditions for convergence of the residual for CDA with noisy partial solution data are the same as the sufficient conditions for convergence of error for CDA using accurate partial solution data.
\end{rem}

\begin{proof}
This proof follows analogously to Theorem \ref{thm:nonoiseresidual} because the terms $I_H(\epsilon_u)$ and $I_H(\epsilon_T)$ in \eqref{CDAnoiseWF} at time $k$ and $k+1$ cancel when subtracting these two systems to form an error equation.
\end{proof}

\begin{theorem}
Let $\mu_1\geq \frac{\nu}{4C_I^2 H^2}$, $\mu_2\geq \frac{\kappa}{4C_I^2H^2}$, $ R_i \nu^{-1}C_p^2<1$, $\sqrt{C_I H}\leq \frac{3}{4}C_p$ and $\frac{4}{3}\sqrt{C_IH}(\alpha_1+\alpha_2)<1$. Then the error in CDA-Picard satisfies
\begin{multline*}
\| u-u^{k+1}\|_\ast + (1-\frac{4}{3} \nu^{-1} R_i C_p^{3/2} (C_I H)^{1/2})\| T-T^{k+1}\|_\ast \leq (\frac{4}{3}(\alpha_1+\alpha_2)( C_I H)^{1/2})^{k+1}\|u-u^0\|_\ast\\+ C_I^2\frac{\frac{4}{3}( C_I H)^{1/2}}{1-\frac{4}{3}(\alpha_1+\alpha_2)( C_I H)^{1/2}}(\mu_1 \|\epsilon_u\| +\mu_2\|\epsilon_T\|),
\end{multline*}
\end{theorem}

\begin{proof}
Let $\ekone=u-u^{k+1}$ and $\etkone=T-T^{k+1}$. We subtract \eqref{BoussWF} from \eqref{CDAnoiseWF} and set $v=\ekone$ and $w=\etkone$. Similarly to Theorem \ref{thm:nonoise} we add and subtract terms and using 
the skew symmetry of $b$ and $\hat{b}$ vanishes four terms, yielding
$$\begin{cases}
\nu\|\nabla \ekone\|^2 +\mu_1\|I_H(\ekone)\|^2&= (R_i (0~ \etkone)^T,\ekone)-b(\ek, u,\ekone)-\mu_1(I_H(\epsilon_u),I_H(\ekone)),\\
\kappa \|\nabla \etkone\|^2+ \mu_2\|I_H(\etkone)\|^2&=-\hat{b}(\ek,  T,\etkone)-\mu_2(I_H(\epsilon_T),I_H(\etkone)).
\end{cases}$$
We use Lemma \ref{starbound}, \eqref{bbound}, Lemma \ref{BoussStability} to upper bound the right hand side as
\begin{equation}\label{noise1}
\begin{cases}
\nu\|\nabla \ekone\|^2 +\mu_1\|I_H(\ekone)\|^2&\leq R_i C_p^{3/2}(C_I H)^{1/2}\| \etkone\|_\ast \| \ekone\|_{\ast}+\alpha_1\nu ( C_I H)^{1/2}\|\ek\|_\ast \|\ekone\|_\ast\\
&~~~+\mu_1\|I_H(\epsilon_u)\| \|I_H(\ekone)\|,\\
\kappa \|\nabla \etkone\|^2 + \mu_2\|I_H(\etkone)\|^2&\leq \alpha_2\kappa ( C_I H)^{1/2}\|\ek\|_\ast \|\etkone\|_\ast + \mu_2\|I_H(\epsilon_T)\| \|I_H(\etkone)\|.
\end{cases}
\end{equation}
Using the interpolation bound and Lemma \ref{starbound} we can upper bound the terms
\begin{align*}
\|I_H(\epsilon_u)\| \|I_H(\ekone)\|&\leq C_I^2 ( C_I H)^{1/2} \|\epsilon_u\| \|\ekone\|_\ast,\\
\| I_H(\epsilon_T)\| \|I_H(\etkone)\|& \leq C_I^2 ( C_I H)^{1/2}\| \epsilon_T\| \|\etkone\|_\ast.
\end{align*}

Combining this with \eqref{noise1} gives
\begin{equation}\label{CDA4}
\begin{cases}
\nu\|\nabla \ekone\|^2 +\mu_1\|I_{H}(\ekone)\|^2&\leq R_i C_p^{3/2} (C_IH)^{1/2}\| \etkone\|_\ast \|  \ekone\|_{\ast} +\alpha_1\nu( C_I H)^{1/2}\|\ek\|_\ast \|\ekone\|_\ast \\
&~~~+ \mu_1 C_I^2 ( C_I H)^{1/2}\|\epsilon_u\| \|\ekone\|_\ast,\\
\kappa \|\nabla \etkone\|^2 + \mu_2\|I_{H}(\etkone)\|^2 &\leq \alpha_2\kappa ( C_I H)^{1/2}\|\ek\|_\ast \|\etkone\|_\ast+ \mu_2 C_I^2 ( C_I H)^{1/2}\|\epsilon_T\| \|\etkone\|_\ast.
\end{cases}
\end{equation}

Next we lower bound the left hand sides using the interpolation bound and the triangle inequality analogously to Theorem \ref{thm:nonoise}
 and use this in \eqref{CDA4} to get that
$$\begin{cases}
\| \ekone\|_\ast &\leq  \frac{4}{3} \nu^{-1} R_i C_p^{3/2} (C_I H)^{1/2}\| \etkone\|_\ast+\frac{4}{3}\alpha_1( C_I H)^{1/2}\|\ek\|_\ast+ \frac{4}{3}\nu^{-1}\mu_1 C_I^2 ( C_I H)^{1/2} \|\epsilon_u\| ,\\
 \| \etkone\|_\ast&\leq  \frac{4}{3}\alpha_2 ( C_I H)^{1/2}\|\ek\|_\ast  + \frac{4}{3}\mu_2 \kappa^{-1}C_I^2 ( C_I H)^{1/2}\|\epsilon_T\|. 
\end{cases}$$

Finally adding the equations gives 
\begin{multline*}
\| \ekone\|_\ast + (1-\frac{4}{3} \nu^{-1} R_i C_p^{3/2} (C_I H)^{1/2})\| \etkone\|_\ast \leq \frac{4}{3}(\alpha_1+\alpha_2)( C_I H)^{1/2}\|\ek\|_\ast + \frac{4}{3} C_I^2 ( C_I H)^{1/2}(\mu_1 \nu^{-1}\|\epsilon_u\| + \mu_2 \kappa^{-1}\|\epsilon_T\| ).
\end{multline*}
Recall that $\frac{4}{3}(\alpha_1+\alpha_2)( C_I H)^{1/2}<1$ therefore
\begin{multline*}
\| \ekone\|_\ast + (1-\frac{4}{3} \nu^{-1} R_i C_p^{3/2} (C_I H)^{1/2})\| \etkone\|_\ast \leq (\frac{4}{3}(\alpha_1+\alpha_2)( C_I H)^{1/2})^{k+1}\|e^0\|_\ast\\+ C_I^2\frac{\frac{4}{3}( C_I H)^{1/2}}{1-\frac{4}{3}(\alpha_1+\alpha_2)( C_I H)^{1/2}}(\mu_1\nu^{-1} \|\epsilon_u\| +\mu_2\kappa^{-1}\|\epsilon_T\|),
\end{multline*}
which completes the proof.

\end{proof}

This result shows that the accuracy of the method depends on the noise of the data given by $\epsilon_u$ and $\epsilon_T$. If $\epsilon_u=0$ and $\epsilon_T=0$ then the result reduces to Theorem \ref{thm:nonoise}. Furthermore, CDA-Picard with noise retains the  improved convergence rate of $H^{1/2}(\alpha_1+\alpha_2)$, but only converges up to the level of the noise.

\section{Numerical Results}\label{Numerical}
We now give numerical results for CDA-Picard for differentially heated cavity problems \cite{Layton89}. We consider \eqref{Bouss} with $f=0$, $g=0$, $\nu=\kappa=10^{-1}$, and $Ri$ is varied which varies $Ra$ for the problems. We will consider two test problems. For the first one, $\Omega=(0,1)^2\subset \R^2$, with boundary conditions
\begin{equation}\label{num:2d}
\begin{cases}
u&= 0 ~\text{ on }\partial\Omega,\\
T(0,y)&=0,\\
T(1,y)&=1, \\
\nabla T\cdot n&=0, \text{ on} ~y=0, ~y=1.
\end{cases}
\end{equation}
We study the convergence of CDA-Picard for $Ra=10000$, $100000$, and $1000000$ with varying $H$ and varying $\mu$. For CDA we will use data with noise and without noise.\\

The discretization for the first test problem is as follows. Let $\tau_h$ be a barycenter refinement mesh with max element diameter $h=1/64$ and note that all results are comparable for other choices of h that we tested. Solutions to this system with $Ra=1000000$ are shown in Figure \ref{2dSol}. We define the spaces $X_h=\mathbb{P}_2\cap X$, $Q_h=\mathbb{P}_1^{disc}(\tau_h)\cap Q$, and $D_h=\mathbb{P}_2(\tau_h) \cap D$. The spaces $(X_h,Q_h)$ satisfy an LBB condition and provide divergence free velocity solutions \cite{arnold:qin:scott:vogelius:2D}, therefore all analytical results hold in the discrete spaces just as in $(X,Q)$.\\

\begin{figure}[H]
\centering
\includegraphics[width=3in, height=3in,clip, trim=4cm 2cm 3cm 1cm]{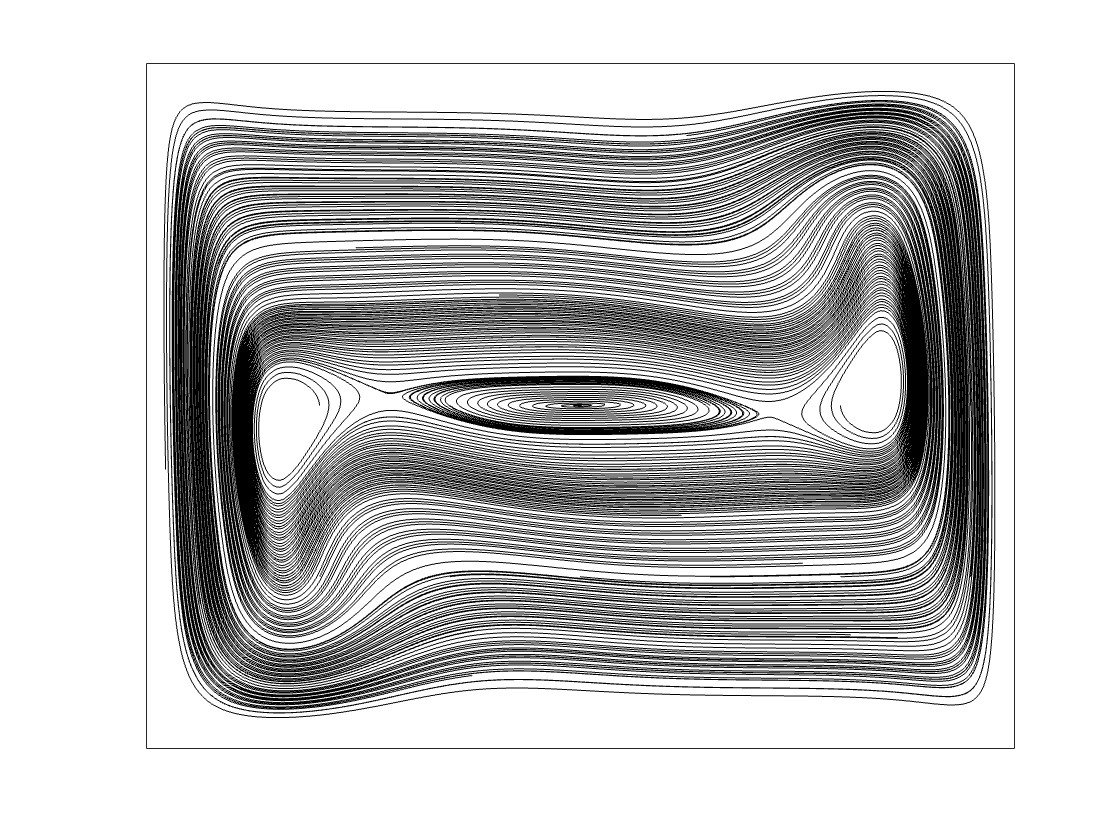}
\includegraphics[width=3.5in, height=3in,clip, trim=6cm 2cm 4cm 1cm]{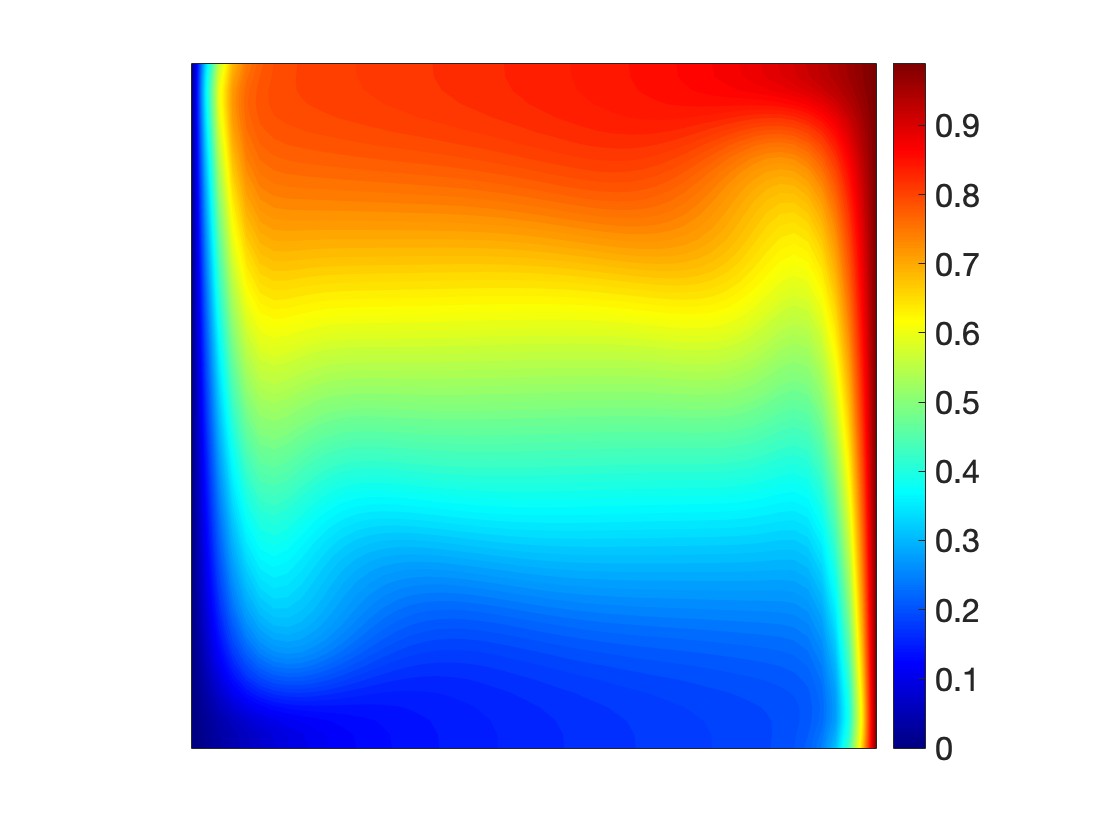}
\caption{Shown above are the computed Boussinesq solution of the differentially heated cavity problem \eqref{num:2d} for velocity streamlines (left) and temperature contours (right) for $Ra=1000000$}\label{2dSol}
\end{figure}

For our second test, we consider the unit cube $\Omega\subset \R^3$ with boundary conditions 
\begin{equation}\label{num:3d}
\begin{cases}
u&= 0 ~\text{ on }\partial\Omega,\\
T(0,y,z)&=0,\\
T(1,y,z)&=1, \\
\nabla T\cdot n&=0, \text{ on} ~y=0, ~y=1,~z=0,~z=1
\end{cases}
\end{equation}
We study the convergence of CDA-Picard for $Ra=10000$, $100000$, and $1000000$ with varying $H$ and varying $\mu$. For CDA we will use data with noise and without noise.\\

For this problem, let $\tau_h$ be a mesh that has been barycenter refined to approximately have a max mesh diameter of $\frac{1}{50}$. Solutions to this system with $Ra=100000$ are shown in Figure \ref{3dSol}. We define a $X_h=\mathbb{P}_3\cap X$, $Q_h=\mathbb{P}_2^{disc}(\tau_h)\cap Q$, and $D_h=\mathbb{P}_3(\tau_h) \cap D$. The spaces $(X_h,Q_h)$ satisfy an LBB condition and provide divergence free velocity solutions \cite{Z10a}, therefore all analytical results hold just as in $(X,Q)$. Recall that Picard decouples the system. Hence we solve the velocity-pressure system using \cite{benzi} which we implement with grad-div stabilization following \cite{HR13}.\\

For CDA, let $\tau_H$ be a mesh with max element diameter $h<H$ whose nodes represent collected data points. Let $I_H$ be the $L^2$ projection operator from $X$ onto $X_H$. This choice of projector allows for CDA implementation by algebraic nudging \cite{RZ21}. We denote the B-norm as
$$\|(u,T)\|_{B}:=\sqrt{\nu \|\nabla u\|^2+\kappa \|\nabla T\|^2}.$$
Having an $H$ dependent norm overcomplicates the results for error and residual, and ultimately we are interested in the $H^1$ convergence results. Therefore we use the $B$-norm in place of the $\ast$-norm and note that in finite dimensional spaces all norms are equivalent.
We use the residual as the stopping criteria
$$\|(u^k,T^k)-(u^{k-1},T^{k-1})\|_{B}=\sqrt{\nu \|\nabla(u^k-u^{k-1})\|^2+\kappa \|\nabla(T^k-T^{k-1})\|^2}<10^{-8}.$$


\begin{figure}[H]
\centering
\includegraphics[width=3.25in, height=3in, clip, trim=3cm 2cm 2cm 3cm]{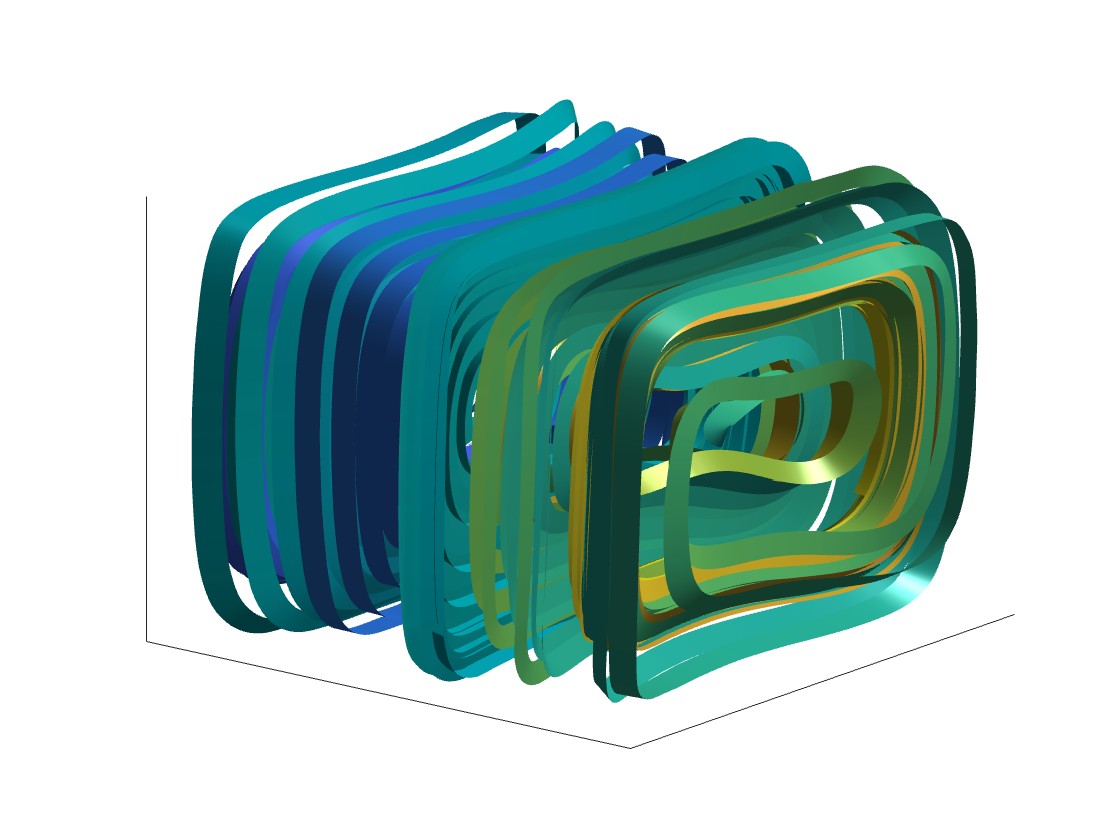}
\includegraphics[width=3.5in, height=3in, clip, trim=0cm 0cm 0cm 1cm]{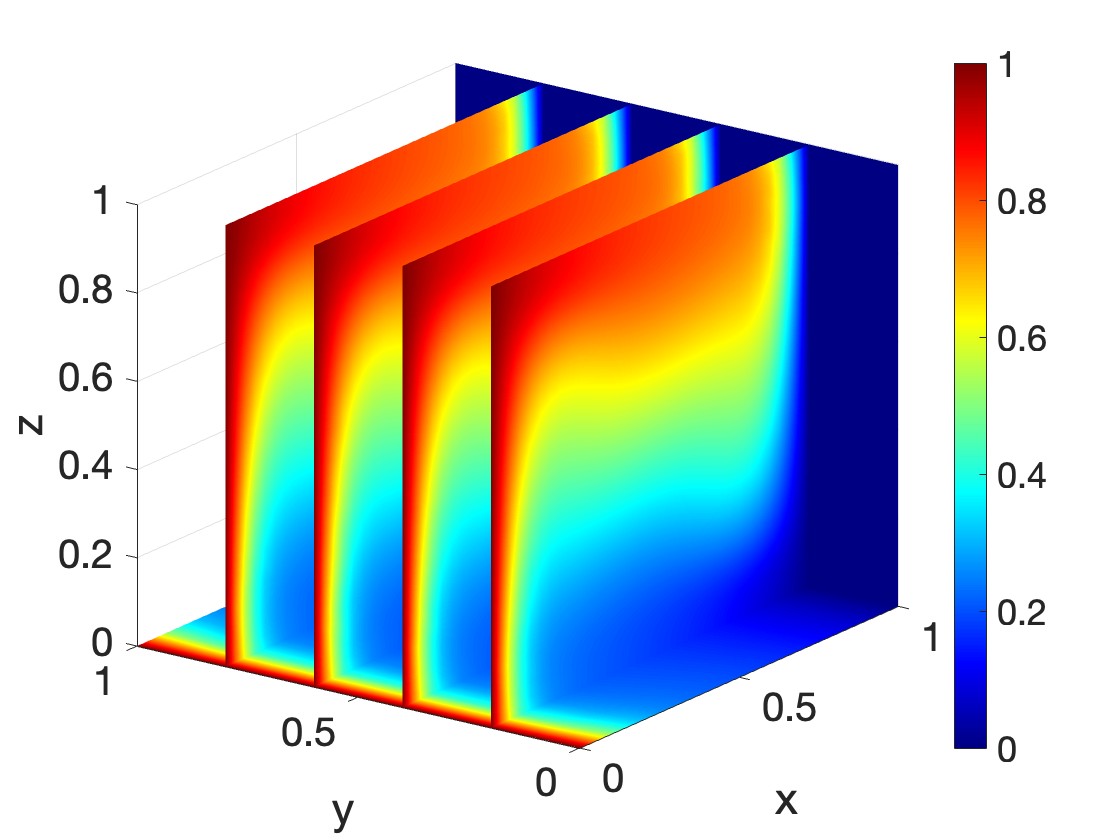}
\caption{Shown above are the computed Boussinesq solution of the differentially heated cavity problem\eqref{num:3d} for velocity streamlines (left) and temperature contours (right) for $Ra=100000$}\label{3dSol}
\end{figure}

\subsection{CDA-Picard without noise}
We begin with CDA-Picard without noise for the heated cavity problem in $2d$ \eqref{num:2d} and give results for $Ra=10000$ and $100000$ with varying $H$ and $\mu_1=\mu_2=1000$. For comparison, we also run Picard without CDA for both and give its results in each plot in the figures. We note that we also performed this test with $\mu_1=\mu_2=1$ and got very similar results for low $Ra$. However, $\mu_1=\mu_2=1000$ performs better for higher $Ra$ and is often necessary for convergence.\\

First, for $Ra=10000$ we see in Figure \ref{fig:nonoise1} (left) that Picard converges but it does so slowly. Immediately when we use CDA-Picard  with $\mu_1=\mu_2=1000$ and $H=1/4$ we see a dramatic improvement in the convergence rate. The amount of data used is relatively small compared to the degrees of freedom (DoF) for the problem. CDA-Picard is further improved as $H$ is decreased fuurther. 
For $Ra=100000$ in Figure \ref{fig:nonoise1} we see that CDA-Picard converges for $\mu=1000$ (right) using $H= \frac{1}{8}, \frac{1}{16}, \text{ and }\frac{1}{32}$ with iteration counts decreasing as $H$ decreases and $\mu$ increases. 


We next provide results for CDA-Picard without noise for the heated cavity problem in $3d$ \eqref{num:3d} for $Ra=10000$ and $100000$ with $\mu_1=\mu_2=1000$ and varying $H$. For comparison, we again run Picard without CDA  and give its results in the figures. For $Ra=10000$ in Figure \ref{3d1e4} (left) we observe that Picard, without CDA, does not converge. In comparison, for CDA-Picard with $H=\frac{1}{5}$ we get convergence. The convergence rate is then improved as $H$ decreases. For $Ra=100000$ we see in Figure \ref{3d1e4} (right) the same behavior with convergence beginning when $H=\frac{1}{10}$.\\ 

\begin{figure}[H]
\centering
\includegraphics[scale=.2]{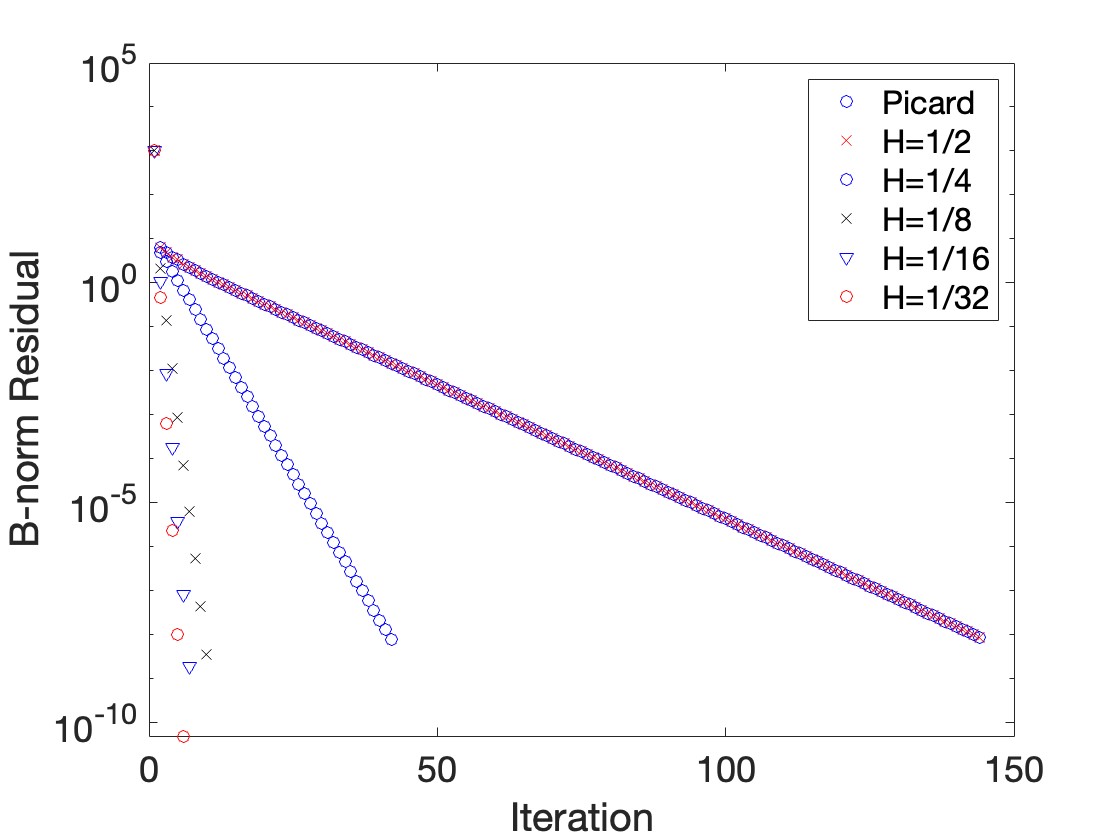}
\includegraphics[scale=.2]{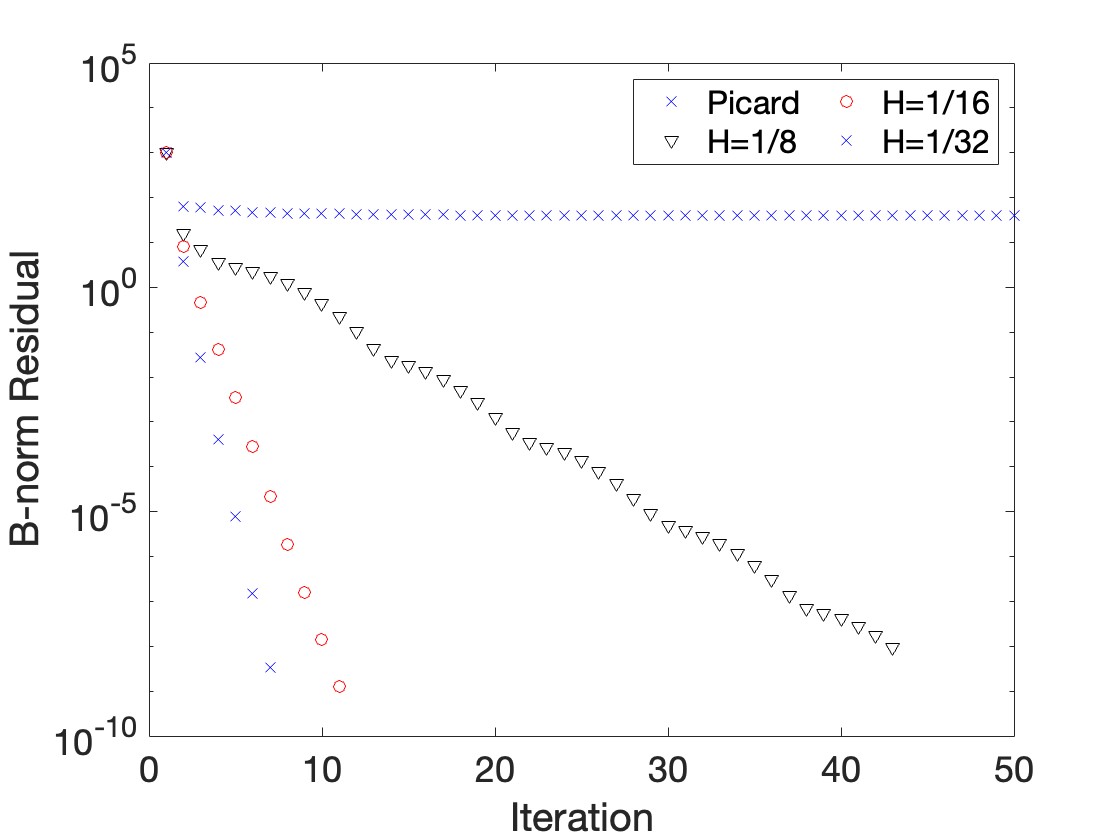}
\caption{Shown above are the convergence plots for the $2d$ heated cavity problem for $Ra=10000$(left) and $Ra=100000$(right) with varying $H$ for $\mu_1=\mu_2=1000.$}\label{fig:nonoise1}
\end{figure}



\begin{figure}[H]
\centering
\includegraphics[scale=.2, clip, trim=0cm 0cm 2cm 0cm]{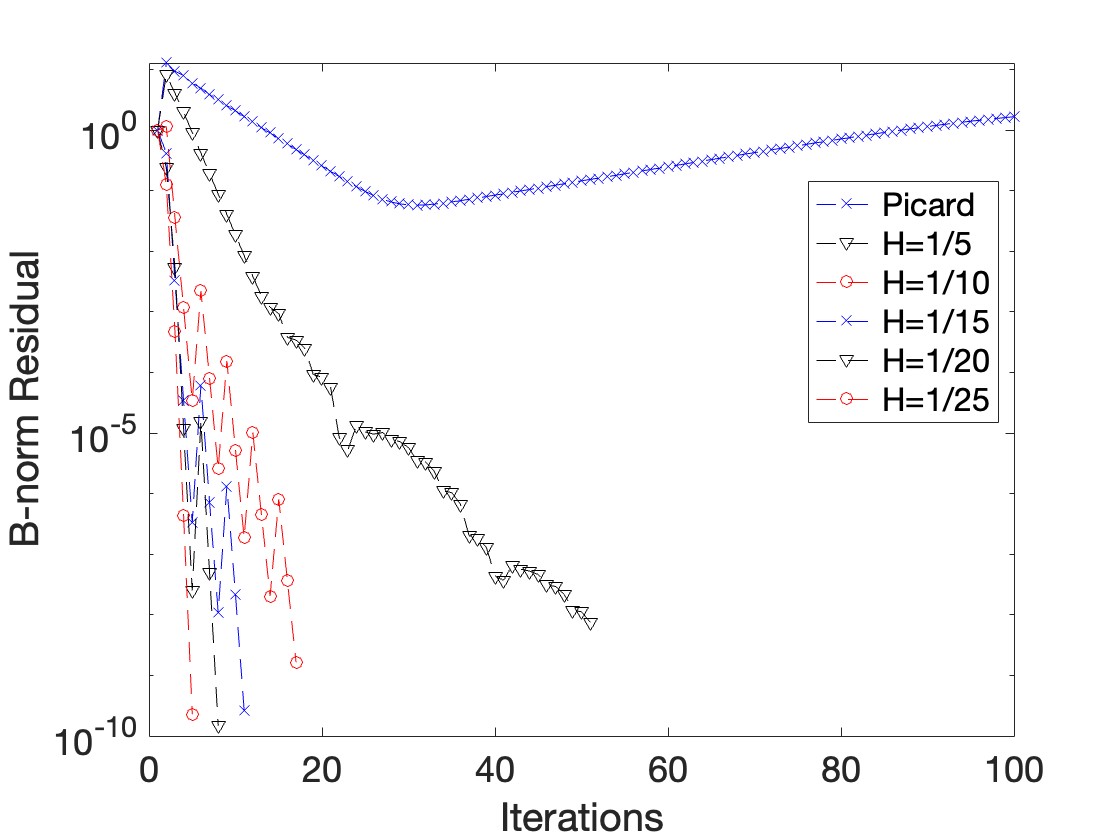}
\includegraphics[scale=.2, clip, trim=0cm 0cm 2cm 0cm]{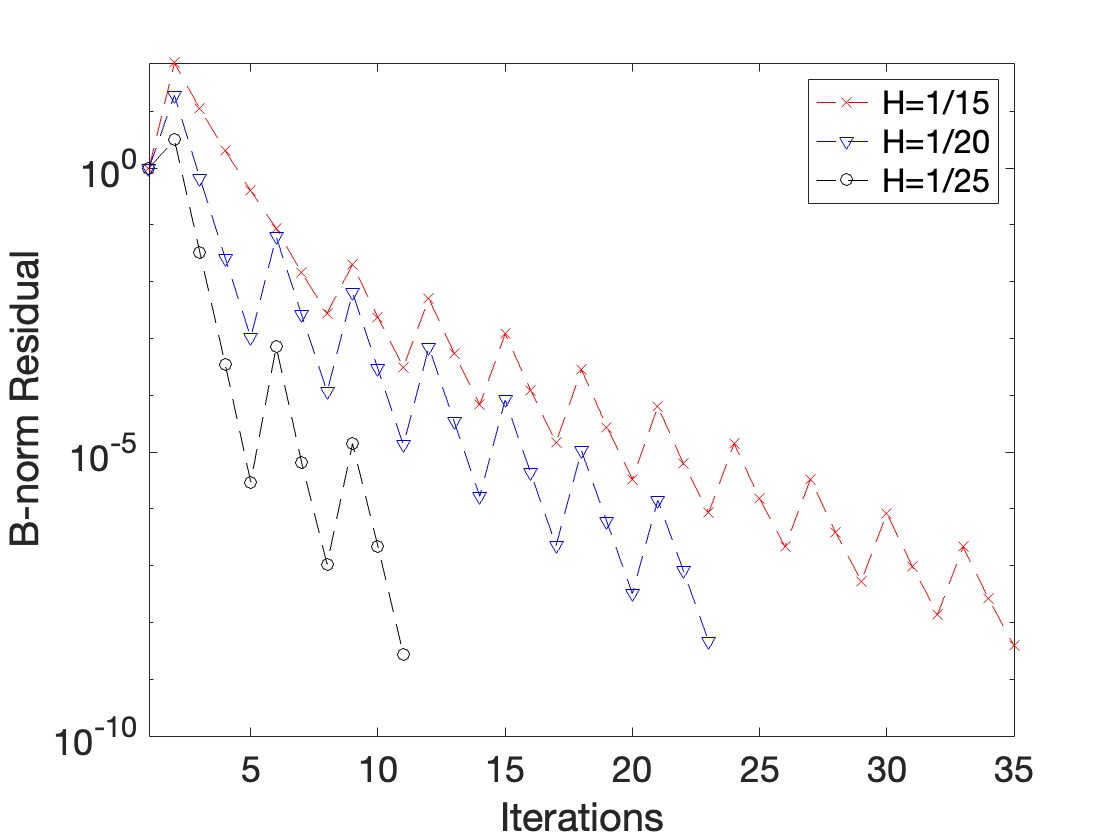}
\caption{Shown above are the convergence plots for the $3d$ heated cavity problem for $Ra=10000$ (left) and $Ra=100000$ (right) with varying $H$ for $\mu_1=\mu_2=1000.$}
\label{3d1e4}
\end{figure}


\subsection{$2d$ heated cavity with CDA using data from only velocity and only temperature}
Next we repeat the same test problem for $Ra=10000$ and $100000$ with $\mu_1=1000, \mu_2=0$ (nudge velocity only) and $\mu_1=0, \mu_2=1000$ (nudge temperature only). For $Ra=10000$ with $\mu_1=1000$ and $\mu_2=0$ we see in Figure \ref{fig:just_u} (left) that CDA Picard converges for $H=1/4, 1/8, 1/16, 1/32$ at approximately the same rate as $\mu_1=\mu_2=1000$. This agrees with our theory. In comparison, for $\mu_1=0$ and $\mu_2=1000$ in Figure \ref{fig:just_T} (left) we see convergence for the same $H$ but it converges slower. 
For $Ra=100000$ with $\mu_1=1000$ and $\mu_2=0$ we see in Figure \ref{fig:just_u} (right) that CDA Picard converges for $H=1/8, 1/16, 1/32$ at a slower rate than $\mu_1=\mu_2=1000$. In comparison, for $\mu_1=0$ and $\mu_2=1000$ in Figure \ref{fig:just_T} (right) we see convergence for the same $H$ but it converges slower. 

\begin{figure}[H]
\centering

\includegraphics[scale=.2]{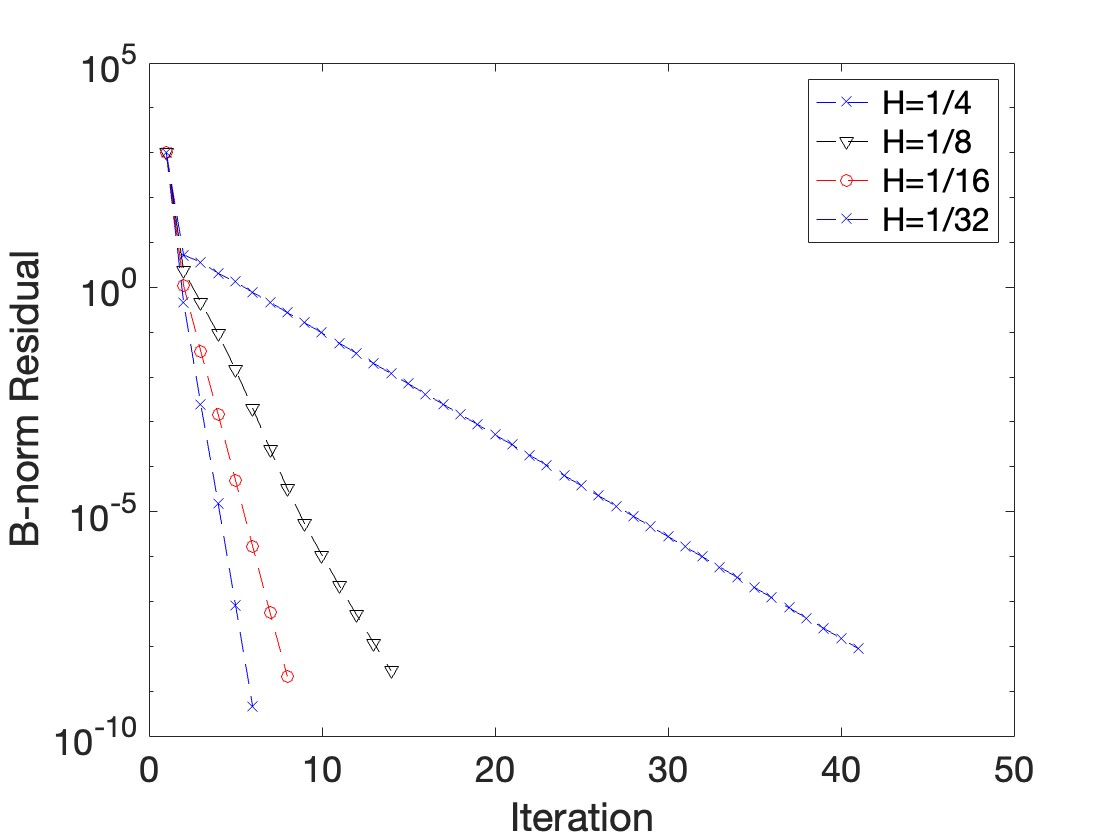}
\includegraphics[scale=.2]{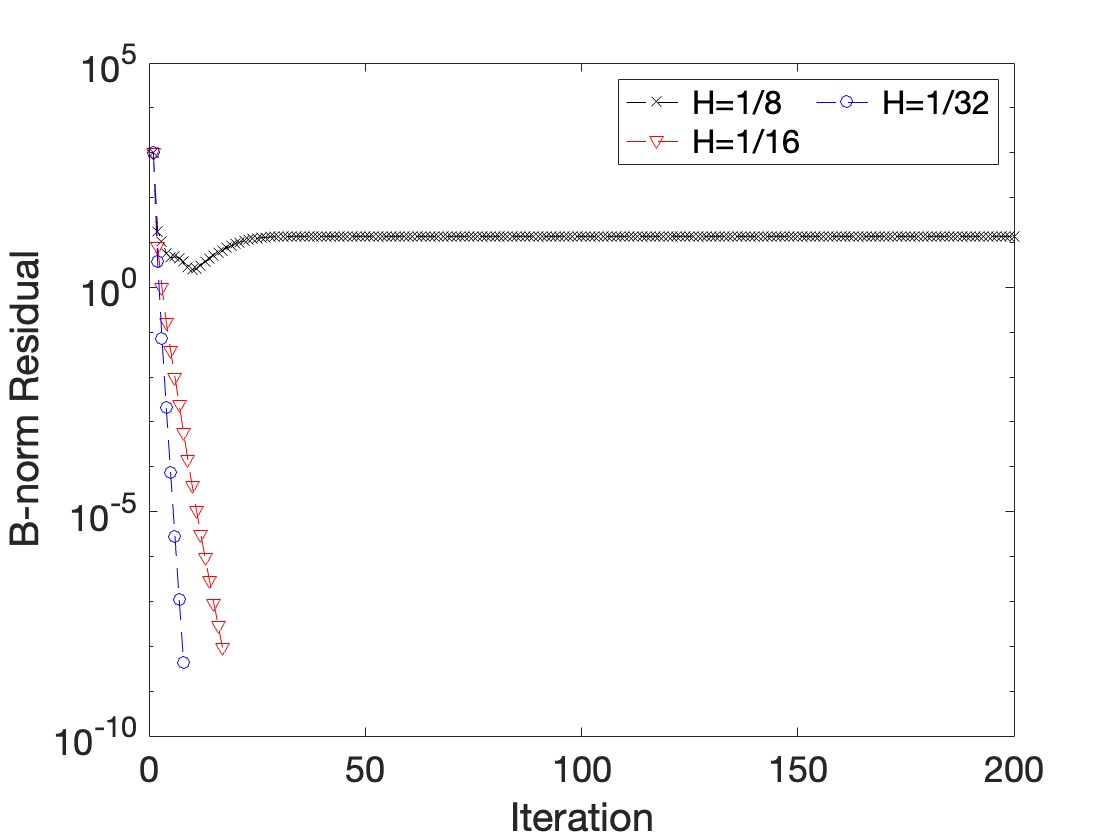}
\caption{Shown above are the convergence plots for $Ra=10000$ (left) and $Ra=100000$ (right) with varying $H$ for $\mu_1=1000, \mu_2=0$}\label{fig:just_u}
\end{figure}


\begin{figure}[H]
\centering
\includegraphics[scale=.2]{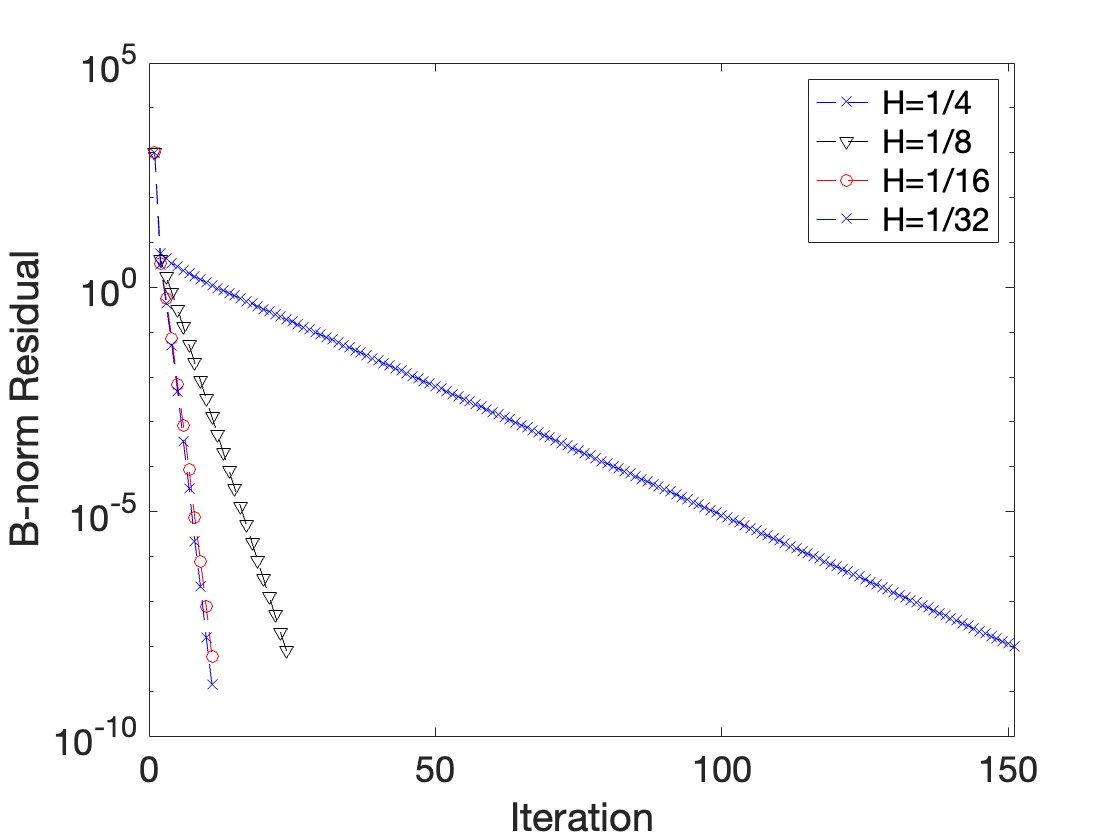}
\includegraphics[scale=.2]{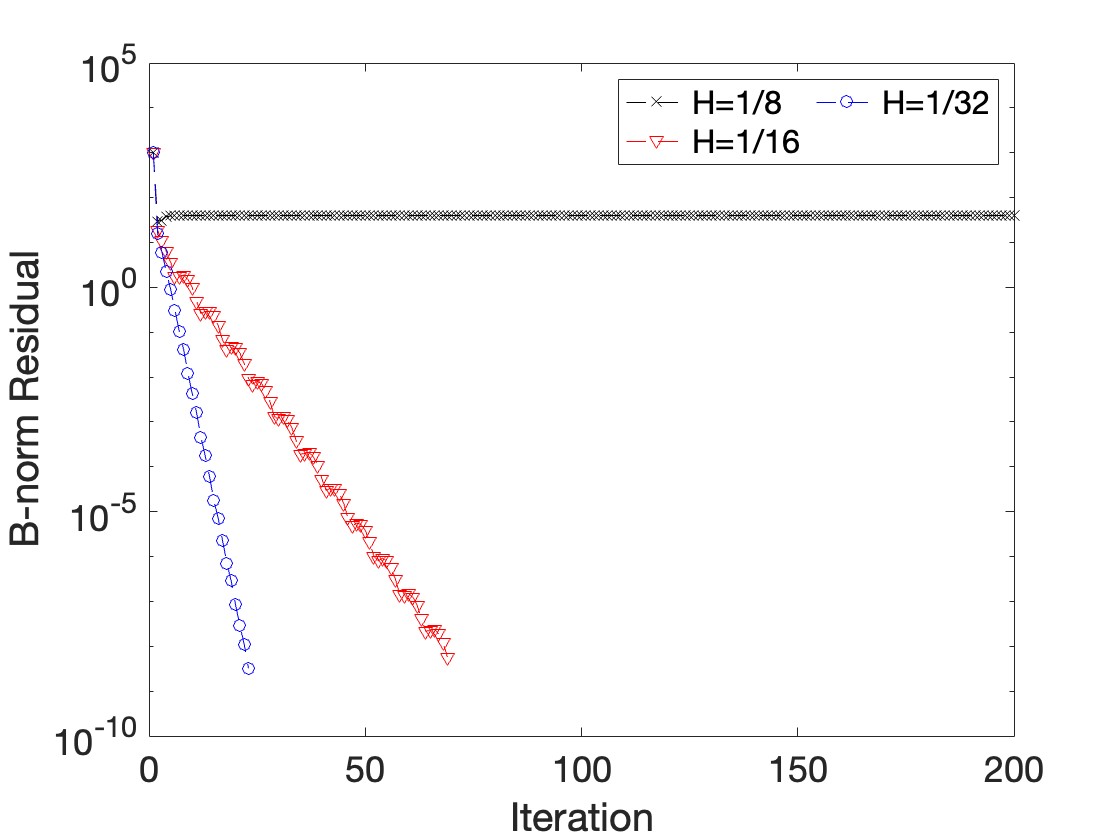}
\caption{Shown above are the convergence plots for $Ra=10000$ (left) and $Ra=100000$ (right) with varying $H$ for $\mu_1=0,\mu_2=1000$}\label{fig:just_T}
\end{figure}


\subsection{CDA-Picard with noise}
We now consider CDA-Picard where the partial solution data has noise. Note that as $\mu_i$ increases the computed solution is increasingly nudged towards the given data. In the case of accurate partial solution data this is desirable but not for noisy partial solution data. Therefore, we consider the case when $\mu_1=\mu_2=1$. We will first consider the $2d$ heated cavity problem for $Ra=100000$ for varying $H$, with noise satisfying $\|\epsilon_u\|,\|\epsilon_T\| = O(10^{-3})$. We compute the noisy partial solution data in the numerical tests by scaling the maximum element of accurate partial solution data by $10^{-3}$, and then adding this to each non-zero term in the partial solution data. For this benchmark test we will provide both error and convergence plots.\\



 Recall that for noise, our theory shows the error of the limit solutoin remains at the same order as the noise. Therefore in order to attain better accuracy using CDA with noisy data we must make an adjustment to the algorithm. When the residual of the iteration is less than $10^{-3}$, CDA-Picard is turned off and the remaining solutions are found using the Newton iteration (without CDA). This is similar to the common method of using Picard to get an initial guess for Newton. However, recall that for these $Ra$ both Picard and Newton fail. Hence CDA-Picard enables Newton to be used.\\
 
 \begin{rem}
 Using this altered approach allows for the use of $\mu_1=\mu_2=1000$ (or larger) because CDA-Picard is not used after a certain accuracy is attained. Therefore the effect of large $\mu_i$ nudging the computed solutions towards the noisy partial solution data is mitigated.
 \end{rem}
 
 For $Ra=100000$ in Figure \ref{fig:noise1e4} we see the residual steadily decreasing while using CDA-Picard. Concurrently, the error is decreasing until it is $O(10^{-3})$. When CDA-Picard is switched to Newton, the residual increases at the first step but immediately decreases quadratically to the desired accuracy. The error for the problem does the same and attains the desired accuracy except the error does not increase at the first time step.\\ 
 
 We now repeat the same problem for heated cavity in $\R^3$ with $Ra=100000$ with $\mu_1=\mu_2=1$, varying $H$, and noise satisfying $\|\epsilon_u\|,\|\epsilon_T\| = O(10^{-3})$. Similarly to the $2d$ tests, we use CDA-Picard until the residual is less than $10^{-3}$ and then switch to Newton (without CDA). We see in Figure \ref{fig:noise1e5} (left) that for $Ra=100000$ CDA-Picard converges with $H=\frac{1}{10}, \frac{1}{15}, \frac{1}{20},$ and $\frac{1}{25}$. When the iterative method changes from CDA-Picard to Newton there is an increase in the residual followed by a decrease for the remainder of the iterations. We also see in Figure \ref{fig:noise1e5} (right) that the errors for CDA-Picard level off at approximately $O(10^{-3})$ for each $H$, which is the same order as the noise.\\ 

\begin{figure}[H]
\centering
\includegraphics[scale=.2]{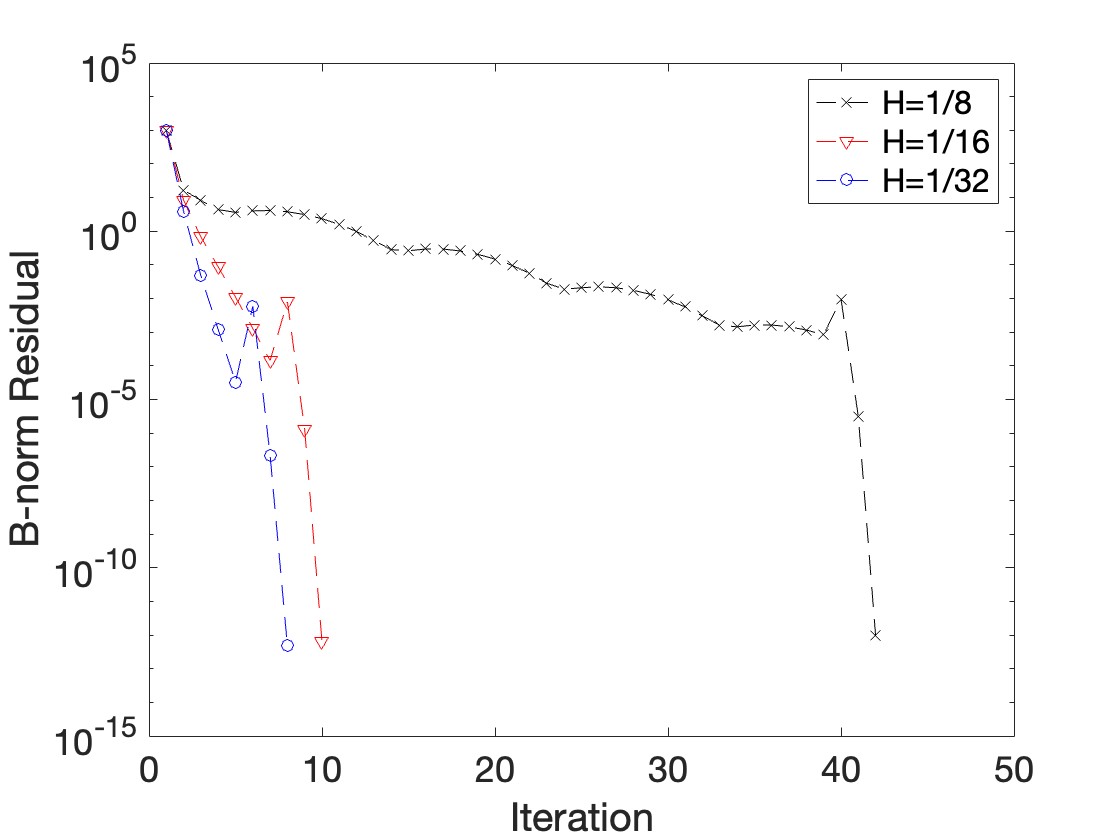}
\includegraphics[scale=.2]{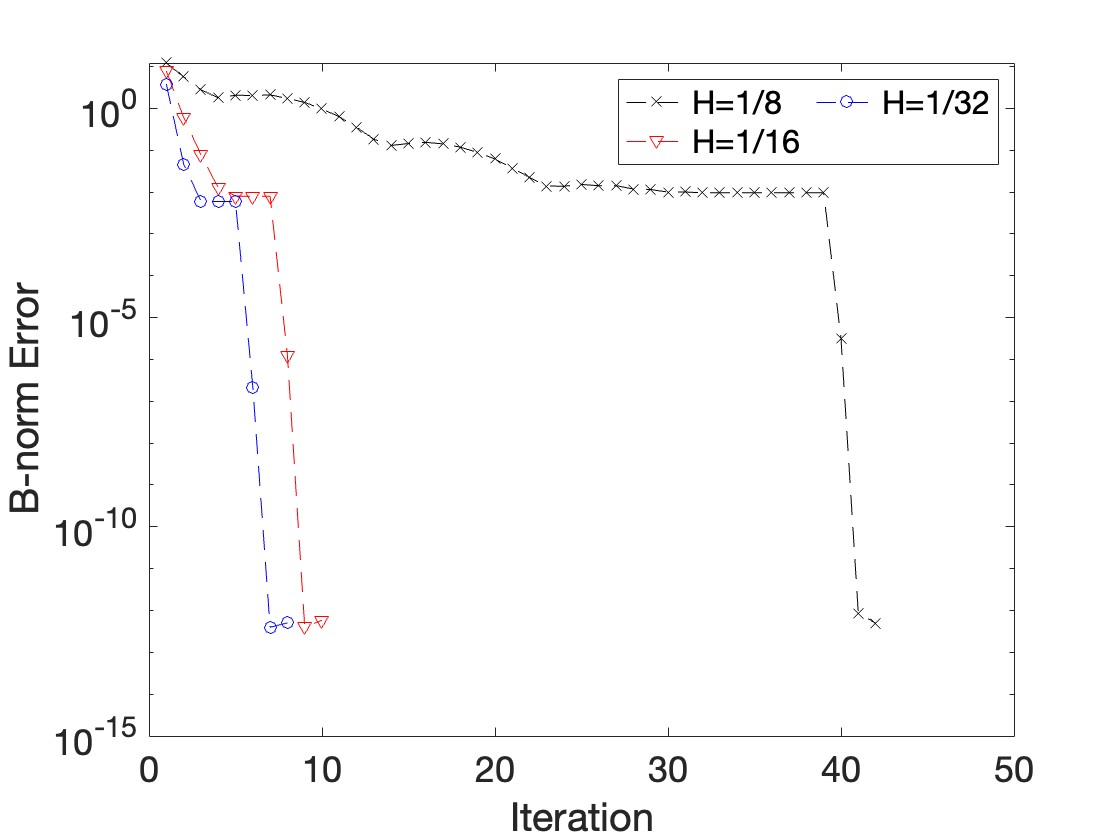}
\caption{Shown above are the convergence plots for the residual (left) and error (right) for the heated cavity problem in $2d$ for $Ra=100000$ with varying $H$, $\mu_1=\mu_2=1$, and signal to noise ratio $10^{-3}$.}\label{fig:noise1e4}
\end{figure}

 
\begin{figure}[H]
\centering
\includegraphics[scale=.2, clip, trim=0cm 0cm 2cm 0cm]{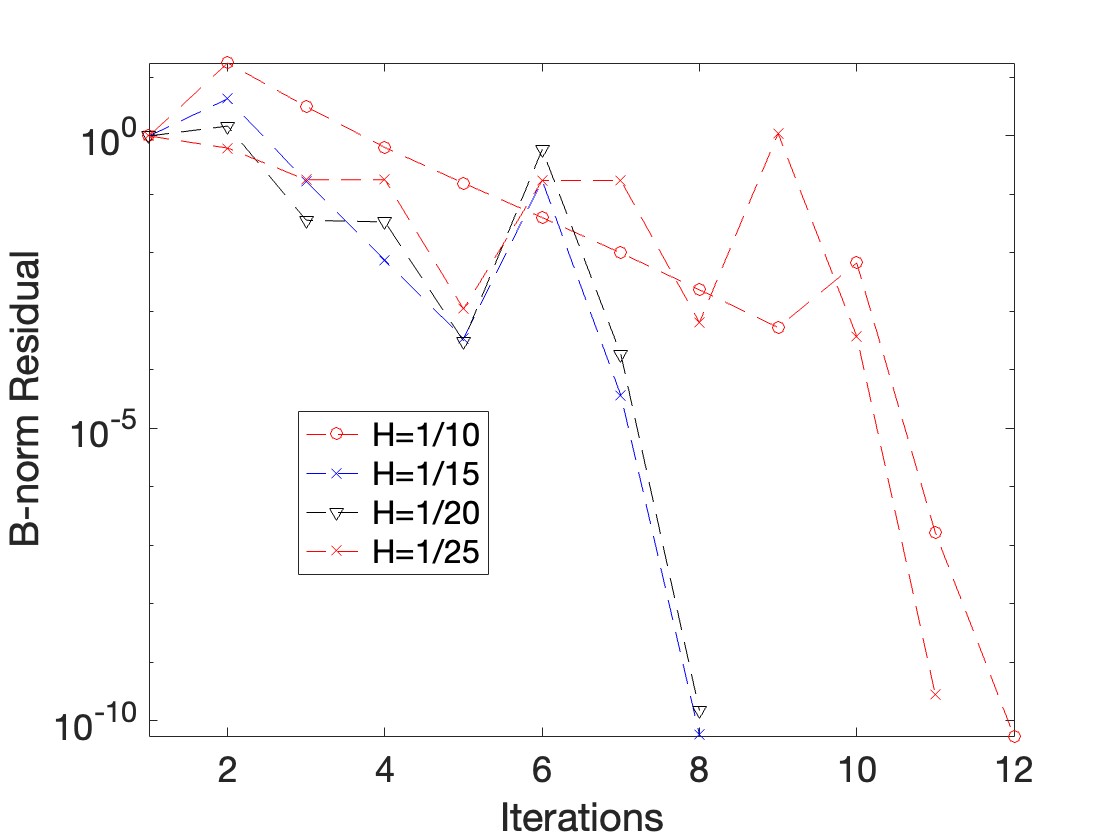}
\includegraphics[scale=.2, clip, trim=0cm 0cm 2cm 0cm]{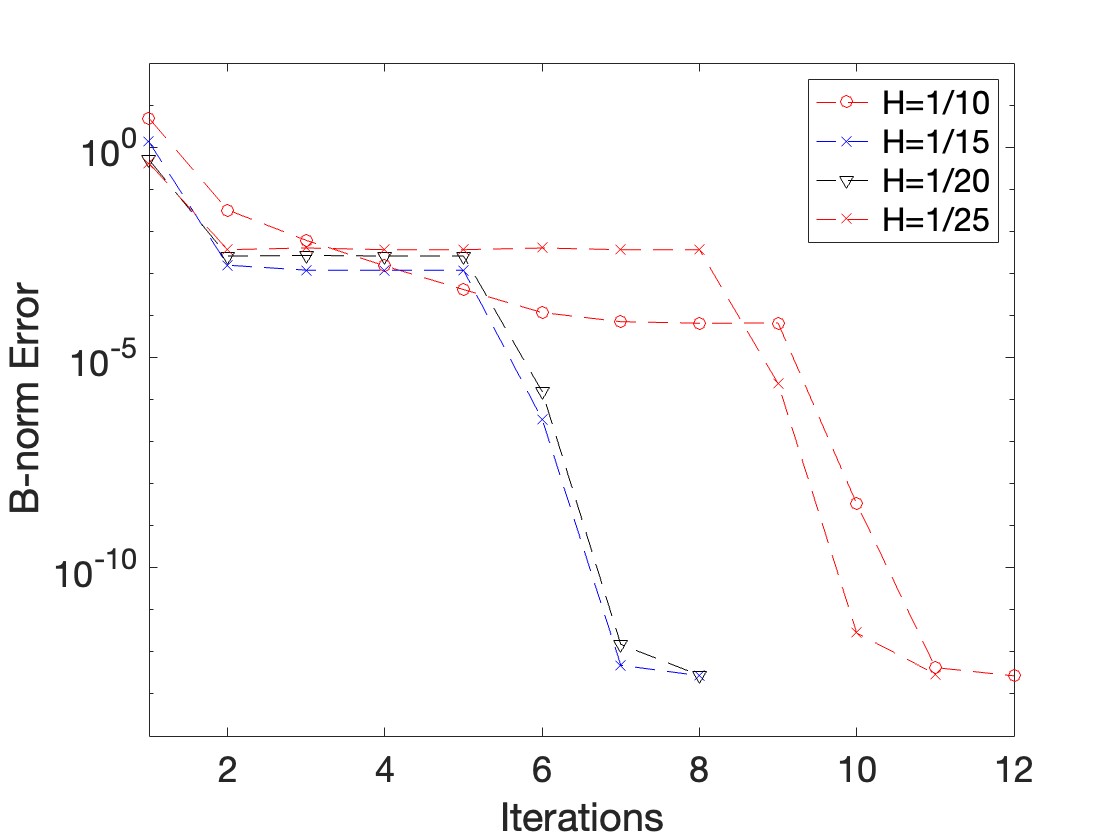}
\caption{Shown above are the convergence plots for the residual (left) and error (right) for the heated cavity problem in $3d$ for $Ra=100000$ with varying $H$, $\mu_1=\mu_2=1$, and signal to noise ratio $10^{-3}$.}
\label{fig:noise1e5}
\end{figure}


\section{Conclusion}
The residual and error of CDA-Picard for the Boussinesq equations converges at a faster rate compared to Picard, with speedup proportional to $H^{1/2}$ when nudging both velocity and temperature or velocity only. Moreover, CDA enables convergence at higher $Ra$. The improved convergence using CDA with no noise on both $u$ and $T$ is demonstrated in both the $2d$ and $3d$ tests where convergence is observed for $Ra$ that fail when using Picard alone. Furthermore the numerical tests  in $2d$ for CDA with no noise on only $u$ and only $T$ also showed improved convergence. 

When the data includes noise, the results are similar with the convergence rate of CDA-Picard containing a multiple of $H^{1/2}$. However, the accuracy of solutions to CDA-Picard is limited by the noise, as one would expect, due to an upper bound on the error containing a term which represents the accuracy of the data. The numerical tests for CDA with noise in both $2d$ and $3d$ demonstrated that the error when using noisy data is limited to the accuracy of the solution data. This was resolved in the numerical tests (achieved convergence) by changing the iteration to Newton once a chosen residual accuracy is attained which shows a possible a solution to the limited accuracy caused by noisy data.

\section{Acknowledgements}
This work was partially supported by the U.S. Department of Energy under award DE-SC0025292.\\ 

\section{Appendix}
\begin{lem}
Any solution to the Boussinesq equations \eqref{BoussWF} satisfies the a priori estimate
\begin{align}
\|\nabla T\|&\leq \|\nabla T\| \leq \kappa^{-1}\|g\|_{D^\ast}=: M_2, \label{TBound}\\
\|\nabla u\|&\leq \|\nabla u\| \leq \nu^{-1}\| f\|_{V^\ast} +R_i C_p^2\nu^{-1} M_2=: M_1\label{uBound}.
\end{align}
\end{lem}

\begin{proof}
We let $v=u$ and $w=T$ in \eqref{BoussWF} then using skew-symmetry gives us
\begin{equation}\label{B1}
\begin{cases}
\nu\|\nabla u\|^2 =Ri( (0~T)^T,u)+(f,u),\\
\kappa\|\nabla T\|^2 = (g,T).
\end{cases}
\end{equation}
We upper bound the right hand side terms using the dual space norms, Cauchy-Schwarz, and Poincar\'e which yields
\begin{align*}
Ri( (0~T)^T,u) &\leq C_p^2 Ri \|\nabla T\| \|\nabla u\|,\\
(f,u) &\leq \|f\|_{V^\ast} \|\nabla u\|, \\
(g,T) &\leq \|g\|_{D^\ast}\|\nabla T\|.
\end{align*}
Using these bounds in \eqref{B1} and reducing provides us with
\begin{equation*}
\begin{cases}
\|\nabla u\| \leq \nu^{-1}\| f\|_{V^\ast} +R_i C_p^2\nu^{-1}\|  \nabla T\|, \\
\|\nabla T\| \leq \kappa^{-1}\|g\|_{D^\ast}, 
\end{cases}
\end{equation*}
and using the second of these bounds in the first,
\begin{equation*}
\|\nabla u\| \leq \nu^{-1}\| f\|_{V^\ast} +R_i C_p^2\nu^{-1} \kappa^{-1}\|g\|_{D^\ast}.
\end{equation*}
This proves the result.
\end{proof}

\begin{lem}
Let $\alpha_1=C_s \nu^{-1} M_1$ and $\alpha_2=C_s \kappa^{-1} M_2$. If $C_p^2 \nu^{-1}R_i , ~C_p^{1/2}(\alpha_2 + \alpha_1)<1$, then the solutions to \eqref{BoussWF} are unique.
\end{lem}
\begin{proof}
Supposing two solutions $(u_1,T_1)$ and $(u_2,T_2)$ to \eqref{BoussWF} exists, and define $e_u= u_1-u_2$ and $e_T=T_1-T_2$. Now subtracting the systems with these two solutions gives $\forall v\in V, w\in D$,
\begin{equation*}
\begin{cases}
b(u_1,e_u,v)+b(e_u,u_2,v)+\nu(\nabla e_u,\nabla v)&= Ri( (0~e_T)^T,v),\\
\hat{b}(u_1, e_T,w)+\hat{b}(e_u,T_2,w)+\kappa(\nabla e_T,\nabla w)&=0.
\end{cases}
\end{equation*}
Taking $v=e_u$ and $w=e_T$ vanishes two nonlinear terms and leaves
\begin{equation*}
\begin{cases}
\nu\|\nabla e_u\|^2&= Ri( (0~e_T)^T,e_u)-b(e_u,u_2,e_u)\\
&\leq C_p^2 R_i \|\nabla e_T\|\|\nabla e_u\| + C_pC_s\|\nabla e_u\|^2\|\nabla u_2\|,\\
&\\
\kappa\|\nabla e_T\|^2&=-\hat{b}(e_u,T_2,e_T)\\
&\leq C_p^{1/2}C_s \|\nabla T_2\| \|\nabla e_u\|\|\nabla e_T\|.
\end{cases}
\end{equation*}
Next, using the bounds \eqref{uBound} and \eqref{TBound} and simplifying gives
\begin{equation}\label{eq:B3}
\begin{cases}
\|\nabla e_u\|&\leq C_p^2 \nu^{-1}R_i \|\nabla e_T\|+ C_p^{1/2}\alpha_1\|\nabla e_u\|,\\
\|\nabla e_T\|&\leq C_p^{1/2}\alpha_2 \|\nabla e_u\|.
\end{cases}
\end{equation}
Adding these and simplifying results in
\begin{equation*}
(1-C_p^{1/2}\alpha_1-C_p^{1/2}\alpha_2)\|\nabla e_u\|+(1-C_p^2 \nu^{-1}R_i )\|\nabla e_T\|\leq 0.\\
\end{equation*}
This provides the uniqueness of the velocity due to the assumption on the data. With this, uniqueness of the temperature follows immediately from the second bound in \eqref{eq:B3}.
\end{proof}

\begin{lem}\
Any solution to the Picard iteration for the Boussinesq equations satisfies the a priori estimate: for any $k=0,1,2...$,
\begin{align*}
 \|\nabla \tkone\|\leq M_2,\\
\|\nabla \ukone\| \leq M_1,
\end{align*}
\end{lem}
\begin{proof}
These results are proved analogously to those of Lemma \ref{BoussStability}.
\end{proof}

\begin{lem}
The Picard iteration \eqref{PicardWF} with data satisfying $\min\left\{1-\nu^{-1}\frac{C_p^2 R_i}{2} ,1-\kappa^{-1}\frac{C_p^2 R_i}{2} \right\}>0$, admits a unique solution. 
\end{lem}

\begin{rem}
Note that \eqref{PicardWF} is linear. This combined with Lemma \ref{Lemma:PicardBound} immediately gives uniqueness. Furthermore, if the domain $\Omega$ is finite dimensional then this also implies existence of solutions.
\end{rem}

\begin{proof}
Let $Y= V\times D$ and at iteration $k+1$ define $A: Y\times Y\rightarrow \R$ and $F: Y\rightarrow \R$ by
\begin{align*}
A( (\hat{u} ,\hat{T}), (v,w)):&=b(\uk , \hat{u}, v) +\nu(\nabla \hat{u},\nabla v) + \hat{b}(\uk , \hat{T}, w) + \kappa (\nabla \hat{T},\nabla w)-R_i ((0 ~\hat{T})^T,v), \\
F((v,w))&= (f,v) + (g,w),
\end{align*}
so that the Picard iteration is given by $A( (\hat{u},\hat{T}), (v,w)) = F((v,w))$. Consider $A( (\hat{u},\hat{T}), (v,w))$. Using \eqref{bbound}, Cauchy-Schwarz, and Young's inequality we lower bound the equation as
\begin{align*}
A( (\hat{u},\hat{T}),(\hat{u},\hat{T})) &= b(\uk, \hat{u}, \hat{u}) +\nu\|\nabla \hat{u} \|^2+ \hat{b}(\uk, \hat{T}, \hat{T}) + \kappa \|\nabla \hat{T}\|^2-R_i ((0 ~\hat{T})^T,\hat{u})  \\
&\geq \nu\|\nabla \hat{u}\|^2 +  \kappa \|\nabla \hat{T}\|^2 -\frac{C_p^2 R_i}{2} \|\nabla \hat{T}\|^2 - \frac{C_p^2 R_i}{2}\|\nabla \hat{u}\|^2 \\
&\geq \min\left\{\nu-\frac{C_p^2 R_i}{2} ,\kappa-\frac{C_p^2 R_i}{2} \right\}\| (\hat{u},\hat{T})\|^2_Y.
\end{align*}
Hence $A$ is coercive. Continuity of $A$ and $F$ follow easily using the bounds and lemmas above. Thus Lax-Milgram applies and gives existence and uniqueness of \eqref{PicardWF}.\\
\end{proof}

\begin{lem}\label{Lemma:PicardConv}
Consider the Picard iteration \eqref{PicardWF} with data satisfying $C_p^2\nu^{-1}R_i<1$ and $C_p^{1/2}(\alpha_1+\alpha_2) <1$. Then the iteration converges linearly with rate $C_p^{1/2}(\alpha_1+\alpha_2)$. In particular we have
\begin{equation*}
 \|\nabla (T-T^{k+1}) \| \leq  C_p^{1/2}\alpha_2 \|\nabla(u-u^k)\|,
\end{equation*}
and
\begin{equation*}
\|\nabla (u-u^{k+1})\| \leq C_p^{1/2}(\alpha_1+ \alpha_2) \|\nabla(u-u^k)\|.
\end{equation*}
\end{lem}

\begin{proof}
Let $\ekone=u-u^{k+1}$ and $\etkone=T-T^{k+1}$. We subtract \eqref{BoussWF} from \eqref{PicardWF} and choose $v=\ekone$ and $w=\etkone$. Using skew-symmetry , vanishes two non-linear terms and leaves the equality
\begin{equation*}
\begin{cases}
b(\ek, u,\ekone) +\nu\|\nabla \ekone\|^2 &= R_i  ((0 ~\etkone)^T,\ekone), \\
\hat{b}(\ek, T,\etkone)+\kappa \|\nabla \etkone\|^2&=0 .
\end{cases}
\end{equation*}
Next we use Cauchy-Schwarz, Poincar\'e, and \eqref{bbound} to upper bound these equations as
\begin{equation*}
\begin{cases}
\nu\|\nabla \ekone\|^2 &\leq C_p^2 R_i \|\nabla\etkone\| \|\nabla\ekone\| + C_pC_s\|\nabla\ek\| \|\nabla u\| \|\nabla\ekone\|,  \\
\kappa \|\nabla \etkone\|^2&\leq C_pC_s\|\nabla\ek\| \|\nabla T\| \|\nabla\etkone\|.
\end{cases}
\end{equation*}
Then we reduce and apply Lemma \ref{BoussStability} to get
\begin{equation*}
\begin{cases}
\|\nabla \ekone\|&\leq C_p^2 \nu^{-1} R_i \|\nabla\etkone\| + C_p\alpha_1\|\nabla\ek\| ,  \\
\|\nabla \etkone\|&\leq C_p\alpha_2 \|\nabla\ek\| .
\end{cases}
\end{equation*}
this gives the bound
\begin{equation*}
 \|\nabla \etkone \| \leq  C_p\alpha_2 \|\nabla\ek\|,
\end{equation*}
Adding the equations and reducing gives
\begin{align*}
\|\nabla \ekone\|+ (1-C_p^2 R_i\nu^{-1})\|\nabla \etkone\|^2 
& \leq C_p(\alpha_1+\alpha_2) \|\nabla\ek\|.
\end{align*}
Finally, using the assumptions on the data finishes the proof.
\end{proof}

\bibliographystyle{plain}
\bibliography{graddiv}

\end{document}